\documentclass[12pt,a4paper,twoside,reqno]{amsart}

\usepackage{hyperref,url} \usepackage{amssymb,latexsym}
\usepackage{amsmath}
\usepackage[latin1]{inputenc} \usepackage{enumerate}
\usepackage[all]{xy} \CompileMatrices \usepackage{amscd}
\usepackage[misc,geometry]{ifsym} 

\usepackage{slashed} 

\def\bysame{\leavevmode\hbox
  to3em{\hrulefill}\thinspace} 

\SelectTips{cm}{12} 
\theoremstyle{plain} \newtheorem{theorem}{Theorem}[section]
\newtheorem{lemma}{Lemma}[section]
\newtheorem{corollary}{Corollary}[section]
\newtheorem{proposition}{Proposition}[section]

\theoremstyle{definition}
\newtheorem{definition}{Definition}[section]
\newtheorem{definition-theorem}{Definition-Theorem}[section]
 \theoremstyle{remark}
\newtheorem{remark}{Remark}[section]

\numberwithin{equation}{section} 

\newcommand{\tr}{\operatorname{tr}}
\newcommand{\pr}{p}
\newcommand{\Id}{\operatorname{Id}} 
\newcommand{\End}{\operatorname{End}}

\newcommand{\Ker}{\operatorname{Ker}}
\newcommand{\ad}{\operatorname{ad}}

\newcommand{\Aut}{\operatorname{Aut}}
\newcommand{\dbar}{\bar{\partial}}  
\newcommand{\CC}{{\mathbb C}} 
 \newcommand{\RR}{{\mathbb R}}
\newcommand{\ZZ}{{\mathbb Z}} \newcommand{\rk}{\operatorname{rk}}
 \renewcommand{\(}{\left(}
  \renewcommand{\)}{\right)} 

\newcommand{\surj}{\to\kern-1.8ex\to}
 
\newcommand{\lra}[1]{\stackrel{#1}{\longrightarrow}}

 \newcommand{\cA}{\mathcal{A}}

 \newcommand{\cM}{\mathcal{M}}
\newcommand{\cP}{\mathcal{P}}

 \newcommand{\cG}{\mathcal{G}}
\newcommand{\cL}{\mathcal{L}}

\newcommand{\Lie}{\operatorname{Lie}}

\newcommand{\cX}{{\widetilde{\mathcal{G}}}}

\def\om{\omega} \def\Om{\Omega}

\def\del{\partial} \def\delb{\overline{\partial}}

\def\Lie{\mathrm{Lie}} \def\Diff{\mathrm{Diff}}
 
  \def\Id{\mathrm{Id}}

\def\cA{\mathcal{A}}  
\def\ctG{\widetilde{\mathcal{G}}}

\def\cG{\mathcal{G}}  \def\cP{\mathcal{P}}

\def\del{\partial} \def\delb{\overline\partial}

\newcommand{\st}{\;|\;}

\newcommand{\ot}{\otimes}

\newcommand{\cCi}{C^\infty}

 \newcommand{\fo}{\mathfrak{o}}

 \newcommand{\fso}{\mathfrak{so}}

\newcommand{\la}{\langle} \newcommand{\ra}{\rangle}

 \newcommand{\Cl}{\mathrm{Cl}}

 \newcommand{\SU}{\mathrm{SU}}
\newcommand{\U}{\mathrm{U}} \newcommand{\GL}{\mathrm{GL}}
\newcommand{\SL}{\mathrm{SL}}

 \newcommand{\OO}{\mathrm{O}}

\newcommand{\chf}{\check{f}} \newcommand{\chg}{\check{g}}

\begin{document}

\title[Moduli, Strominger and Killing spinors in generalized geometry]{Infinitesimal moduli for the Strominger system and Killing spinors in generalized geometry}
\thanks{This project has received funding from the European Union's Horizon 2020 research and innovation programme under the Marie Sklodowska-Curie grant agreement No 655162. This work is partially supported by an ESF - Short Visit Grant 5717 within the framework of the ITGP network. MGF is supported by a Marie Sklodowska-Curie grant and was initially supported by ICMAT Severo Ochoa project SEV-2011-0087 and by the \'Ecole Polytechnique F\'ed\'eral de Lausanne. RR is supported by IMPA and was initially supported by QGM through its partnership with the Mathematical Institute of Oxford. CT is partially supported by Agence Nationale de la Recherche - ANR project EMARKS}

\author[M. Garcia-Fernandez]{Mario Garcia-Fernandez}
\author[R. Rubio]{Roberto Rubio}
\author[C. Tipler]{Carl Tipler}

\address{Instituto de Ciencias Matem\'aticas (CSIC-UAM-UC3M-UCM)\\
  Nicol\'as Cabrera 13--15, Cantoblanco\\ 28049 Madrid, Spain}
\email{mario.garcia@icmat.es} \address{IMPA, Estrada Dona Castorina,
  110, Jardim Bot\^anico, Rio de Janeiro - RJ, 22460-320, Brazil}
\email{rubio@impa.br} \address{LMBA, UMR CNRS 6205; D\'epartement de
  Math\'ematiques, Universit\'e de Bretagne Occidentale, 6, avenue
  Victor Le Gorgeu, 29238 Brest Cedex 3 France}
\email{carl.tipler@univ-brest.fr}



\maketitle

\begin{abstract}
  We construct the space of infinitesimal variations for the
  Strominger system and an obstruction space to integrability, using
  elliptic operator theory.  We initiate the study of the geometry of
  the moduli space, describing the infinitesimal structure of a
  natural foliation on this space. The associated leaves are related
  to generalized geometry and correspond to moduli spaces of solutions
  of suitable Killing spinor equations on a Courant algebroid. As an
  application, we propose a unifying framework for metrics with
  holonomy $\SU(3)$ and solutions of the Strominger system.
\end{abstract}

\tableofcontents

\section{Introduction}
\label{sec:intro}

The Strominger system couples a pair of Hermite--Yang--Mills
connections with a conformally balanced hermitian metric on a
Calabi--Yau threefold $X$, by means of an equation for
$4$-forms---known as \emph{the Bianchi identity}. Although originated
in string theory \cite{HullTurin,Strom}, its mathematical study was proposed by
Yau \cite{Yau1} as a natural generalization of the Calabi problem
\cite{Calabi,Yau0}, in relation to moduli spaces of Calabi-Yau
threefolds which are not necessarily K\"ahlerian.

Pioneered by Fu, Li and Yau \cite{FuYau,LiYau}, the existence problem
for the Strominger system has been an active area of research in
mathematics in the last ten years (see
\cite{AGF1,FeiYau,FIUVa,Fu,FuLiYau,FTY,TsengYau} and references
therein).  There is an important conjecture by Yau \cite{Yau2}, which
states that any stable holomorphic vector bundle $V$ over a
homologically balanced Calabi--Yau threefold $X$ \cite{Michel} with
$c_2(V) = c_2(X)$ admits a solution of the Strominger system.  This
conjecture is widely open, the main difficulties being its
non-K\"ahler nature and the lack of understanding of the geometry of
the equations.

A problem closely related to Yau's conjecture is the construction of a
moduli space of solutions of the Strominger system. This moduli
problem remained almost unexplored for a long time, despite its
interest in string theory, where it describes the most basic pieces
(scalar massless fields) of the four-dimensional 
theory induced by a 
heterotic string compactification. Indeed, only a few references that
tackle the preliminary question of constructing the tangent space at a
given solution can be found in the physics literature
\cite{AGS,BeckerTseng,BeckerTsengYau,CyLa,MelSha,OssaSvanes}.  This
first step turns out to be rather challenging, and a complete answer
has been so far elusive.

The prime motivations for this work are the construction of the moduli
space of solutions of the Strominger system and its interrelation with
Yau's conjecture. In this paper we make a contribution to the first
problem, constructing the space of infinitesimal variations of a
solution and an obstruction space to integrability. We initiate the
study of the geometry of the moduli space, describing the
infinitesimal structure of a natural foliation, whose leaves are
intimately related to generalized geometry \cite{Hit1}. By
investigating the tangent to a leaf, we give an interpretation of the
leaves as moduli spaces of solutions of suitable Killing spinor
equations on a Courant algebroid. This last tangent space arises
naturally as a quotient of a bigger finite-dimensional vector space by
the second de Rham cohomology group of $X$. Our construction provides
a unifying framework for metrics with holonomy $\SU(3)$ and solutions
of the Strominger system, that we expect will have future applications
to Yau's conjecture. To explain our results, let us first recall the
definition of the equations.


\subsection*{Background}

Let $(X,\Omega)$ be a Calabi-Yau threefold, that is, a complex
manifold of dimension three endowed with a nowhere vanishing
holomorphic section of the canonical bundle $\Omega \in
H^0(X,\Lambda^{3,0}T^*)$. We do not assume that $X$ is
K\"ahlerian. Let $P_K$ be a principal bundle over $X$ with compact
structure group $K$.  The Strominger system
is 
\begin{equation}\label{eq:stromintro}
  \begin{split}
    F^{0,2}=0, \ \ \  F\wedge \omega^2&=0,\\
    R^{0,2}=0, \ \ \  R\wedge \omega^2&=0,\\
    d^* \omega - i(\dbar - \partial)\log \|\Omega\|_\omega & = 0,\\
    dd^c \omega - \alpha'(\tr R\wedge R - \tr F\wedge F)&=0,
  \end{split}
\end{equation}
with unknowns given by a hermitian metric $g$ on $X$ with fundamental
form $\omega$, a connection $A$ on $P_K$ and a metric connection
$\nabla$ on the (smooth) tangent bundle of $X$. Here, $\alpha'$ is a
positive real constant and $F$ and $R$ denote, respectively, the
curvature $2$-forms of $A$ and $\nabla$. The notation $-\tr$ refers to
the Killing form on the Lie algebra of $K$.  In this paper we impose
that $\nabla$ is unitary with respect to the hermitian structure
$(\Omega,\omega)$. 

The Strominger system comprises, essentially, three conditions---the
first two well understood in the literature. First, the equation in
the third line, often known as the \emph{dilatino equation}, is
strongly reminiscent of the complex Monge-Amp\`ere equation on a
K\"ahler manifold (see e.g. \cite{FWW}). It restricts the holonomy of the Bismut connection
$$
\nabla^+ = \nabla^g - \frac{1}{2}g^{-1}(d^c\omega)
$$
to $\SU(3)$, where $\nabla^g$ denotes the Levi-Civita connection of
the metric $g$. Furthermore, as observed by Li and Yau \cite{LiYau},
the dilatino equation is equivalent to the condition
\begin{equation}\label{eq:confbalanced}
  d(\|\Omega\|_\omega \omega^2) = 0,
\end{equation}
which implies that $\omega' = \|\Omega\|_\omega^{1/2}\omega$ is the
fundamental form of a balanced metric, namely $d^*\omega' = 0$, and
hence $g$ is \emph{conformally balanced}. A classical result of
Michelsohn \cite{Michel} characterizes the existence of balanced
metrics on a complex manifold using a condition on the
homology---formulated in terms of currents and known as the
\emph{homologically balanced condition}.

Second, the first two lines in \eqref{eq:stromintro} correspond to the
Hermite--Yang--Mills condition for the connections $A$ and $\nabla$
with respect to the conformally balanced metric $g$. There is a
well-known theory for Hermite--Yang--Mills connections on a hermitian
manifold $(X,g)$ \cite{lt}, which ranges from existence results to the
construction of the moduli space (which turns out to be K\"ahler when
$g$ is conformally balanced). The main result of the theory is
Li--Yau's theorem \cite{LiYauHYM}, which characterizes the existence of
solutions in terms of (slope) stability of the bundle, generalizing
the Donaldson--Uhlenbeck--Yau theorem in K\"ahler geometry
\cite{Don,UY}.

Finally, the most demanding and less understood condition of the
system is the \emph{Bianchi identity}
\begin{equation}\label{eq:bianchiintro}
  dd^c\omega = \alpha'(\operatorname{tr} R \wedge R - \operatorname{tr} F \wedge F),
\end{equation}
which is ultimately responsible for the non-K\"ahler nature of the
problem. The non-vanishing of the \emph{Pontryagin term}
$\operatorname{tr} R \wedge R - \operatorname{tr} F \wedge F$ prevents
the hermitian form $\omega$ to be closed and hence allows the complex
manifold $X$ to be non-K\"ahlerian. This subtle condition, which
arises in the quantization of 
the physical theory, was studied by Freed \cite{Freed} in the context
of index theory for Dirac operators and more recently by Sati--Schreiber--Stasheff from the point of view of \emph{differential string structures} \cite{SSS}. Despite these important topological
insights, we have an almost total lack of
understanding of this last equation from an analytical point of view.

\subsection*{Main results}

In this work we add to the understanding of the moduli problem for the
Strominger system. The first contribution of this work is a complete
and direct construction of the vector space of infinitesimal
variations of a given solution---the infinitesimal moduli space---
using an elliptic complex $S^*$.

\begin{theorem}
  \label{th:intro1}
  The space of infinitesimal deformations of solutions of the
  Strominger system is given by the first cohomology group $H^1(S^*)$
  of an elliptic complex of multi-degree differential operators
  $S^*$. This complex admits a natural extension $\tilde S^*$, and the space of obstructions is defined as $H^2(\tilde S^*)$.
\end{theorem}

To clarify the exposition, we first undertake the construction of the
complex for a toy model in Section \ref{sec:infmoduliabelian}. For
this, we introduce an abelian version of the equations
\eqref{eq:stromintro} depending on a real
parameter. 
The analysis in the abelian case will show that the combination of the
Bianchi identity with the conformally balanced equation
\eqref{eq:confbalanced} is well-behaved at the level of symbols.

In Section \ref{sec:infmoduligeneral} we construct the elliptic
complex of differential operators $S^*$ and identify its first
cohomology
$$
H^1(S^*)
$$
with the infinitesimal moduli of solutions of the Strominger
system. Some of the difficulties that arise in the construction of $S^*$ come from the symmetries of the system, which
turn out to have a Lie groupoid structure due to the compatibility of
the connection $\nabla$ with the metric $g$.

In Section \ref{sec:anomalyflux} we investigate the geometry on the
moduli space of solutions of the Strominger system $\mathcal{M}$
derived from the Bianchi identity.
This moduli
space 
is endowed with a canonical $H^3(X,\RR)$-valued closed $1$-form
$$
\delta \in \Omega^1(\mathcal{M},H^3(X,\RR))
$$
which is constructed via the variation of
\eqref{eq:bianchiintro}. 
The kernel of $\delta$ defines an integrable distribution on the
tangent bundle of $\mathcal{M}$ and hence a foliation on the moduli
space. A striking fact about this foliation is that its leaves can be
understood by using Hitchin's theory of generalized geometry
\cite{Hit1}.  The aim of Section \ref{sec:anomalyflux} is to give a
rigorous account of the infinitesimal version of this picture.  The
construction of a differentiable structure on $\mathcal{M}$ and the local structure of the foliation will be addressed in future work.

Neglecting obstructions to integrability, the tangent to a leaf at a
point is defined by an exact sequence
\begin{equation*}
  \xymatrix{
    0 \ar[r] & H^1(\mathring{S}^*) \ar[r] & H^1(S^*) \ar[r]^{\delta} & H^3(X,\RR). \\
  }
\end{equation*}
As the notation suggests, $H^1(\mathring{S}^*)$ is the cohomology of a
complex which, surprisingly, needs 
generalized geometry 
for its rigorous definition. In this new framework
$H^1(\mathring{S}^*)$ has a natural interpretation, as variations of a
suitable generalized metric modulo generalized diffeomorphisms. A
special feature of $H^1(\mathring{S}^*)$ is that it cannot be
constructed by standard elliptic operator theory, as the space of
generalized vector fields cannot be identified with the space of
global sections of a vector bundle (similarly as the space of
symplectic vector fields on a symplectic manifold). Motivated by this
problem, we construct a refinement of $H^1(\mathring{S}^*)$, which
fits into the following exact diagram
\begin{equation}\label{eq:SQDiagram2}
  \xymatrix{
    & 0  \ar[d]   &  &\\
    & H^2(X,\RR)  \ar[d]  &  &\\
    & H^1(\widehat{S}^*)  \ar[d]  & &\\
    0 \ar[r] & H^1(\mathring{S}^*) \ar[d] \ar[r] & H^1(S^*) \ar[r]^{\delta} & H^3(X,\RR). \\
    & 0 &  &  &
  }
\end{equation}
Unlike $H^1(\mathring{S}^*)$, the refined vector space
$H^1(\widehat{S}^*)$ is constructed by considering \emph{inner
  symmetries} of a smooth, transitive, Courant algebroid, and is
defined as the first cohomology of an elliptic complex of degree $1$
differential operators. Motivation for the previous construction comes
from two basic principles in the physics of the heterotic string,
given by the \emph{Green-Schwarz mechanism} \cite{GreenSchwarz}, and
the \emph{flux quantization condition}. We should stress that the
space $H^1(\widehat{S}^*)$ is the one that comes closer to the physics
of the heterotic string.

Section \ref{sec:stgen} gives a geometric interpretation of the leaves
of the foliation in the moduli space, showing the strong connection
between the Strominger system and generalized geometry. For this, we
define suitable Killing spinor equations
\begin{equation}\label{eq:Killingintro}
  D^\phi_+ \eta = 0, \qquad \slashed D_-^\phi \eta = 0,
\end{equation}
for an admissible metric on a smooth transitive Courant algebroid and
prove the following result.

\begin{theorem}\label{th:charStrom}
  The Strominger system \eqref{eq:stromintro} is equivalent to the
  Killing spinor equations \eqref{eq:Killingintro}, on a transitive
  Courant algebroid obtained from reduction. As a consequence,
  \eqref{eq:stromintro} is a natural system of equations in
  generalized geometry, that is, solutions are exchanged under
  generalized diffeomorphisms.
\end{theorem}

This result builds on previous work of the first author in the
relation between generalized geometry and heterotic supergravity
\cite{GF}.  Theorem \ref{th:charStrom} gives a precise interpretation
of the vector space $H^1(\mathring{S}^*)$, as infinitesimal
deformations for solutions to the Killing spinor equations
\eqref{eq:Killingintro} modulo infinitesimal symmetries of the Courant
algebroid. As a consequence, a leaf of the foliation determined by
$\delta$ can be interpreted as a moduli space of solutions of these
equations.

The proof of Theorem \ref{th:charStrom} reveals a strong parallelism
with the theory of metrics with holonomy $\SU(3)$, once generalized
geometry enters into the game. The same equations, formulated instead
on an exact Courant algebroid, pin down precisely Riemannian metrics
with holonomy $\SU(3)$ on a six dimensional manifold.  Generalized
geometry provides a unifying framework for the theory of the
Strominger system and the well-established theory of Calabi--Yau
metrics, which we expect will have interesting applications in the
former. We will explore further this analogy in future work.

In the physics literature, de la Ossa and Svanes \cite{OssaSvanes} and
Anderson, Gray and Sharpe \cite{AGS} have recently proposed a close
approximation to the infinitesimal moduli for the Strominger system in
terms of the Dolbeault cohomology of a holomorphic double
extension---based on previous ideas by Melnikov and Sharpe
\cite{MelSha}. In a sequel to the present paper we will show that
their proposal admits a natural interpretation in generalized
geometry, and relates in a certain way to the infinitesimal moduli of
the Strominger system.  After Section
\ref{sec:stgen} 
was completed, we were informed by A. Coimbra that an alternative
formulation of the Strominger system using generalized geometry was
provided recently in the physics literature \cite{CoMiWa}. Our
approach has the benefit of making evident that the Strominger system
is invariant under generalized diffeomorphisms.

\textbf{Acknowledgments:} We thank Luis \'Alvarez-C\'onsul, Bjorn
Andreas, Ves\-tis\-lav Apostolov, Henrique Bursztyn, Ryushi Goto,
Marco Gualtieri, Nigel Hitchin, Laurent Meersseman, Xenia de la Ossa,
Dan Popovici, Brent Pym and Eirik Svanes for useful discussions. Part
of this work was undertaken while CT was visiting IMPA, UFRJ, CRM,
during visits of MGF and CT to CIRGET, and of RR to EPFL and ICMAT. We
would like to thank these very welcoming institutions for providing a
nice and stimulating working environment.

\section{Infinitesimal moduli: abelian
  case}\label{sec:infmoduliabelian}
\label{sec:infabel}
The aim of this section is to study a toy model for the construction
of the infinitesimal moduli space of the Strominger system, which
avoids the difficulties arising from the treatment of the unitary
connection on the tangent bundle and non-abelian groups. In
particular, we will consider deformations of a Calabi-Yau structure on
a compact, six dimensional, smooth manifold $M$, endowed with a
holomorphic line bundle. We do not require our complex manifolds to be
K\"ahlerian.

\subsection{The abelian equations}
\label{sec:abelianeq}
Let $M$ be a compact, oriented, six dimensional smooth manifold. Let
$L$ be a hermitian line bundle over $M$. We fix a non-zero real
constant $c$. We denote by $T$ the smooth tangent bundle of $M$ and
its complexification by $T_\CC$. Consider triples
$(\Omega,\theta,\omega)$ where $\Omega$ is a complex $3$-form such
that
\begin{equation}
  \label{eq:T01}
  T^{0,1}:=\lbrace V\in T_\CC \;\vert\; \iota_V \Om =0 \rbrace
\end{equation}
determines an almost complex structure $J_\Omega$ on $M$, $\theta$ is
a unitary connection on $L$, and $\omega$ is a $J_\Omega$-compatible
$2$-form, that is,
\begin{equation}\label{eq:Jomegacompatible}
  \omega(J_\Omega \cdot,J_\Omega \cdot) = \omega \quad \textrm{and} \quad \omega(\cdot,J_\Omega \cdot) > 0 
\end{equation}
is a Riemannian metric on $M$. We aim to construct a space of
infinitesimal deformations for solutions of the equations
\begin{equation}\label{eq:stromabelian}
  \begin{split}
    d\Om & =0, \ \ \ \ \ \ \ \ \ \ \ \ \ \ \ d(\vert\vert \Om \vert \vert_\om \om^2)  =  0,\\
    F_\theta^{0,2} & = 0, \ \ \ \ \ \ \ \ \ \ \ \ \ \ \ \ F_\theta \wedge \omega^2 = \lambda \omega^3,\\
    dd^c \omega - c(F_\theta\wedge F_\theta) & = 0,
  \end{split}
\end{equation}
where $\lambda \in i \RR$ is a constant that depends on the unitary
structure determined by $(\Omega,\omega)$ and the first Chern class of
$L$, as follows from the identity
\begin{equation}\label{eq:degree}
  \operatorname{deg} (L)  : = c_1(L)\cdot [\|\Omega\|_\omega \omega^2] = \frac{i\lambda}{2\pi}\int_M \|\Omega\|_\omega \omega^3.
\end{equation}
Our convention for the point-wise norm of a $(3,0)$-form $\varphi$
with respect to $\omega$ is
\begin{equation}\label{eq:norm3form}
  \|\varphi\|^2_\omega \omega^3 = 6i\varphi \wedge \overline{\varphi}. 
\end{equation} 
Recall that the integrability of $J_\Omega$ is equivalent to the
condition
$$
d \Omega = 0
$$
and hence, by the first two equations in \eqref{eq:stromabelian}, any
solution determines a Calabi-Yau threefold structure on $M$ endowed
with a holomorphic line bundle.

When $c_1(L) = 0$ and $c = -1$ the system \eqref{eq:stromabelian}
corresponds to the field equations of \emph{abelian heterotic
  supergravity} considered in \cite{MaSp} (note that in our notation
$F_\theta$ is a purely imaginary $2$-form).
The reason why we do not work directly with this case, which is the
situation that comes closer to the Strominger system, is the following
observation, that shall be compared with \cite[p. 55]{CHSW}.

\begin{proposition}

  Let $(\Omega,\theta,\omega)$ be a solution of
  \eqref{eq:stromabelian}. If $c_1(L) = 0$ and $(M,\Omega)$ is a $\partial \dbar$-manifold,
then $\theta$ is flat and $\omega$ has holonomy $\SU(3)$ (in particular, $\omega$ is K\"ahler Ricci-flat).
\end{proposition}
\begin{proof}

As $c_1(L) = 0$,
$F_\theta$ is exact. By the $\partial \dbar$-lemma, $F_\theta
= \partial \dbar f$ for some smooth function $f$ on $X$. After
conformal re-scaling of the hermitian metric on $L$, we obtain a flat
Chern connection on the holomorphic line bundle
$(L,\dbar_\theta)$. Then, since $\theta$ is the Chern connection of a
hermitian-Einstein metric on $L$, $\theta$ has to be flat by
uniqueness, and hence it follows that $\omega$ is strong K\"ahler with torsion.
By the conformally balanced condition, the Bismut connection of $\omega$ has holonomy $\SU(3)$ (see~\cite[Section~II]{Strom}). Applying
now~\cite[Corollary~4.7]{IvanovPapadopoulos}, $\omega$ is K\"ahler and
the result follows.
\end{proof}

In this work we are mainly interested in non-K\"ahler solutions of
\eqref{eq:stromabelian}, and therefore we will assume that $c_1(L) \neq 0$. Non-K\"ahler
solutions of \eqref{eq:stromabelian} can be obtained using the
perturbative method in \cite{AGF1}, from holomorphic line bundles $L$
over a projective Calabi-Yau threefold $X$ with non-torsion $c_1(L)
\in H^2(X,\ZZ)$ satisfying
$$
c_1(L)^2 = 0.
$$

\begin{remark}
  It is perhaps more natural to consider the Hermite-Yang-Mills
  equations in \eqref{eq:stromabelian} with respect to the conformally
  balanced metric $\|\Omega\|^{\frac{1}{2}}_\omega\omega$. The
  linearization of these alternative abelian equations is, however,
  more involved and does not add to the understanding of the
  Strominger system.
\end{remark}

\subsection{Notation, parameter space and symmetries}
\label{sec:setup}
Let $\Om$ be a complex $3$-form on $M$ such that \eqref{eq:T01}
determines an almost complex structure $J$ on $M$ with
anti-holomorphic tangent bundle $T^{0,1}$.  Denote by $\Omega^k$
(resp. $\Omega^k_\CC$) the space of real (resp. complex) smooth $k$-forms
on $M$. We denote by $\Omega^{p,q}$ the space of $(p,q)$-forms on
$(M,J)$ and by $\Omega^{p,q}(W)$ the space of $(p,q)$-forms taking
values in a vector bundle $W$. Given a $q$-form $\sigma$ on $M$ (that
may take values in a vector bundle) and $\gamma\in\Om^{p}(T_\CC)$, we
define a $(q+p-1)$-form $\sigma^\gamma$ by
\begin{equation}\label{eq:sigmagamma}
  \sigma^\gamma= (\sigma(\gamma,\cdot))^{skw}
\end{equation}
where $skw$ denotes the skew-symmetric part of the tensor
$\sigma(\gamma,\cdot)$ satisfying :
$$
\forall (V_j)\in (T_\CC)^{q+p-1}\; , \;
\sigma(\gamma,\cdot)(V_1,\ldots,V_{q+p-1})=\sigma(\gamma(V_1,\cdots,V_p),V_{p+1},\ldots,V_{q+p-1}).
$$
As $\Om$ is nowhere vanishing, it induces an isomorphism on forms
\begin{equation}\label{eq:T0isom}
  \begin{array}{cccc}
    T_0: & \Om^{0,j}(T^{1,0}) & \rightarrow & \Om^{2,j} \\
    & \gamma & \mapsto & \Om^\gamma.
  \end{array}
\end{equation}
We will also denote by $T_0$ the induced isomorphism in cohomology
\cite{popo}.

To study the infinitesimal moduli of \eqref{eq:stromabelian} we define
the following parameter space.  Let $\cA$ be the space of unitary
connections on $L$ and $\Omega^3_0 \subset \Omega^3_\CC$ the
non-linear subspace of complex $3$-forms such that \eqref{eq:T01}
determines an almost complex structure on $M$.  We set
\begin{equation}\label{eq:spaceP}
  \cP:=\lbrace (\Om,\theta,\om)\;\vert\; \om \text{ is } J_\Om-\text{compatible} \rbrace \subset \Omega^3_0 \times \cA \times \Omega^2,
\end{equation}
where the compatibility condition is as in
\eqref{eq:Jomegacompatible}.

Let $\operatorname{Diff}_0$ be the identity component of the group of
diffeomorphisms of $M$.  Consider the group $\ctG$ of automorphisms of
$L$ that preserve the unitary bundle structure and cover an element in
$\operatorname{Diff}_0$.  This group preserves $\cP$, exchanging
solutions of \eqref{eq:stromabelian}. Denote by $\cG$ the gauge group
of the hermitian bundle $L$. Then we have an exact sequence
\cite{ACMM}
\begin{equation}
  \label{eq:sequencegroup}
  1\to \cG \lra{} \ctG \lra{\pr} \operatorname{Diff}_0 \to 1.
\end{equation}
Given a connection $\theta$ on $L$, we have a lift $
\Omega^0(T)\rightarrow \Lie(\ctG)$ and at the level of vector spaces
the corresponding Lie algebra sequence splits
\begin{equation}
  \label{eq:sequenceLie}
  0\to \Lie(\cG) \lra{} \Lie(\ctG) \lra{\pr} \Omega^0(T) \to 0.
\end{equation}

\subsection{Linearization and ellipticity}
In the sequel, we fix a solution $(\Omega,\theta,\omega)$ of
\eqref{eq:stromabelian}. The integrable almost complex structure
determined by $\Omega$ will be denoted by $J$ and the curvature of
$\theta$ will be denoted by $F$.
\label{sec:linear}
The complex encoding infinitesimal deformations of
\eqref{eq:stromabelian} is built from an elliptic complex
parameterizing infinitesimal deformations of the complex structure
that preserve the Calabi-Yau condition, that is, with trivial
canonical bundle, \cite{Got}
\begin{equation}
  \label{eq:CYinfinitesimal}
  0\rightarrow \Om^0(T^{1,0}) \lra{\cL_{\cdot}\Om} \Om^{3,0}\oplus \Om^{2,1}\lra{d} \Om^{3,1}\oplus\Om^{2,2}.
\end{equation}
Here, the first non-trivial arrow is defined by the infinitesimal
action of $\operatorname{Diff}_0$, given by the Lie derivative of
$\Om$ and we use the following characterization of the tangent space
of $\Omega^3_0$ at $\Omega$:
$$
T_{\Omega} \Omega^3_0 = \Om^{3,0}\oplus \Om^{2,1}.
$$
The variations in $\Om^{3,0}$ 
are just rescaling of the holomorphic $3$-form, while elements in $
\Om^{2,1}$ correspond via $T_0$ to deformations of the complex
structure. Given $\dot \Om\in \Om^{3,0}\oplus \Om^{2,1}$, we will
denote by $\dot J$ the associated variation of almost complex
structure given by
\begin{equation}\label{eq:dotJtensor10}
  \dot J^{1,0} = 2i T_0^{-1}(\dot\Om^{2,1}),
\end{equation}
with $\dot J^{1,0}=\frac{1}{2}(\dot J -i J\dot J).$

Let $T_0\cP$ be the tangent space of $\cP$ at the initial solution:
$$
T_0\cP =\lbrace (\dot\Om,\dot \theta, \dot \om)\in
\Om^{3,0}\oplus\Om^{2,1}\oplus \Om^1(i\RR)\oplus \Om^2 \;\vert\; J
\dot \om - \dot \om=-\om^{\dot J J} \rbrace.
$$
Note that the equations that define $T_0\cP$ can be equivalently
written as
\begin{equation}\label{eq:dotomega02}
  2i \dot \om^{0,2} = \om^{\dot J^{1,0}}
\end{equation}
and therefore there is a canonical isomorphism
\begin{equation}\label{eq:T0Psplit}
  T_0\cP \cong A^1:= \Om^{3,0}\oplus\Om^{2,1}\oplus \Om^1(i\RR)\oplus \Om^{1,1}_\RR,
\end{equation}
where $\Om^{1,1}_\RR \subset \Om^{1,1}$ denotes the space of real
$(1,1)$-forms on $(M,\Omega)$, given explicitly by
\begin{equation}\label{eq:isoomega}
  \dot \omega = \dot \omega^{1,1} + \frac{1}{2}\omega^{\dot J J}.
\end{equation}
Consider the linearization of the equations (\ref{eq:stromabelian})
$$
\mathbf{L} \colon T_0\cP \rightarrow A^2 := \Omega^{3,1} \oplus
\Omega^{2,2} \oplus \Omega^{0,2} \oplus \Omega^4 \oplus \Omega^5
\oplus \Omega^6(i\mathbb{R}).
$$
Using the vector space splitting \eqref{eq:sequenceLie} given by the
fixed connection $\theta$, the infinitesimal action

$$\mathbf{P}\colon \Lie(\ctG)\rightarrow T_0\cP$$
reads explicitly
\begin{equation}
  \label{eq:abelianinfinaction}
  \begin{array}{cccc}
    \mathbf{P} \colon &A^0 : = \Omega^0(T)\times  \Om^0(i\RR) & \rightarrow &  T_0\cP \\
    &  (V,r) & \mapsto & (d\iota_{V^{1,0}} \Om,  \iota_V F+dr, \cL_V \om)
  \end{array}
\end{equation}
with $V^{1,0}=\frac{1}{2}(V-iJV)$. We construct a complex of
differential operators
\begin{equation}
  \label{eq:abcomplex}
  (A^*) \qquad \qquad \qquad  A^0 \lra{\mathbf{P}}  A^1  \lra{\mathbf{L}} A^2,
\end{equation}
combining the operators $\mathbf{P}$ and $\mathbf{L}$ with the
isomorphism \eqref{eq:T0Psplit}. Our aim is to prove that this complex
is elliptic. Note that an arbitrary unitary connection on $L$ is of
the form $\theta + \dot \theta$ where $\dot \theta \in\Om^1(i\RR)$,
with corresponding curvature 
$F + d \dot \theta$, and also that $\lambda = \lambda(\Omega,\omega)$
is a function of the hermitian structure given by
\eqref{eq:degree}. Using this, we obtain the following expression for
the differential $\mathbf{L}$, regarded as an operator with domain
$T_0\cP$.
\begin{lemma}
  \label{lem:diff}
  The differential $\mathbf{L} = \oplus_{i=1}^5 \mathbf{L}_i$ is given
  by
  \begin{equation}
    \label{eq:diff}
    \begin{array}{ccc}
      \mathbf{L}_1(\dot \Om,\dot \theta,\dot\om) & = & d\dot \Om,  \\
      & & \\
      \mathbf{L}_2(\dot \Om,\dot \theta,\dot\om) & = & \dbar \dot \theta^{0,1} + \frac{i}{2}(F^{\dot J})^{0,2},\\
      & & \\
      \mathbf{L}_3(\dot \Om,\dot \theta,\dot\om) & = & d\(J d \dot \om - J (d\om)^{\dot J J} - 2c(\dot \theta\wedge F)\),\\
      & & \\
      \mathbf{L}_4(\dot \Om,\dot \theta,\dot\om) & = & d\left(2\vert\vert\Om\vert\vert_{\om}\dot\om\wedge\om+ 
        \delta(\vert\vert \Om\vert\vert_\om) \om^2\right), \\
      & & \\
      \mathbf{L}_5(\dot \Om,\dot \theta,\dot\om) & = & d\dot \theta\wedge \om^2 + (2 F - 3 \lambda \omega)\wedge \dot\om\wedge \om - \dot \lambda \omega^3.
    \end{array}
  \end{equation}
where $\dot J$ and $\dot \lambda$ are, respectively, the
  infinitesimal variations of almost-complex structure and constant
  $\lambda$ defined by $\dot \Om$ and $\dot \om$, and
  \begin{equation}\label{eq:dotnormOmega}
    \delta (\vert\vert \Om\vert\vert_\om) = \vert\vert \Om\vert\vert_\om^{-1}\operatorname{Re}(\dot\Omega,\Omega)_\omega - \frac{1}{2} \vert\vert \Om\vert\vert_\om \Lambda_{\omega}\dot \omega.
  \end{equation}

\end{lemma}
\begin{proof}
  The calculation of $\mathbf{L}_1$, $\mathbf{L}_4$ and $\mathbf{L}_5$
  is straightforward and for the calculation of $\mathbf{L}_2$ see
  e.g. \cite{GFT}. Formula \eqref{eq:dotnormOmega} follows from
  \eqref{eq:norm3form}.  To compute $\mathbf{L}_3$, note that
  $d^c=JdJ^{-1}$, where the action of $J$ on forms $b\in \Om^p$ is
 $$Jb=b(J^{-1}\cdot,\ldots,J^{-1}\cdot)=(-1)^p b(J\cdot,\ldots,J\cdot).$$
 Then, using the compatibility between $\omega$ and $J_\Omega$ in
 \eqref{eq:stromabelian} we have $d^c_{J_\Omega}\omega = J_\Omega d
 \omega$ and therefore
 \begin{equation}
   \nonumber
   \frac{d}{dt}_{|t= 0} (J_t dJ_t^{-1} \om_t) =  \frac{d}{dt}_{|t= 0} (J_t(d\om)) + Jd \dot \om.
 \end{equation}
 Lastly,
 \begin{eqnarray}
   \nonumber
   \frac{d}{dt}_{|t= 0} (J_t d \om)
   &=&
   -d\om(\dot J\cdot,J\cdot,J\cdot)-d\om( J\cdot,\dot J\cdot,J\cdot)-d\om(J\cdot,J\cdot,\dot J\cdot)
   \\
   \nonumber
   &=&
   d\om(\dot JJ^2\cdot,J\cdot,J\cdot)+d\om( J\cdot,\dot JJ^2\cdot,J\cdot)+d\om(J\cdot,J\cdot,\dot JJ^2\cdot)
   \\
   \nonumber
   &=&
   -J(d\om(\dot JJ\cdot,\cdot,\cdot)+d\om(\cdot,\dot JJ\cdot,\cdot)+d\om(\cdot,\cdot,\dot JJ\cdot))
   \\
   \nonumber
   &=&
   -J(d\om)^{\dot J J}.
   \nonumber
 \end{eqnarray}
\end{proof}

We note that the differential operator $\mathbf{L}$ is of first order
in the components $\mathbf{L}_1, \mathbf{L}_2, \mathbf{L}_4$ and
$\mathbf{L}_5$, but $\mathbf{L}_3$ has order two. We shall use the
generalized notion of ellipticity provided by Douglis and Nirenberg
\cite{DN}. For the general theory of linear multi-degree elliptic
differential operators we refer to \cite{LM1,LM2}. Here we recall the
basic definition. Let $E$ and $F$ be smooth real vector bundles over
the compact manifold $M$ with a direct sum decomposition
$$
E = \bigoplus_{j=1}^m E_j, \qquad F = \bigoplus_{i=1}^l F_i,
$$
and $\mathbf{L} \colon \Omega^0(E) \to \Omega^0(F)$ a linear
differential operator with corresponding decomposition $\mathbf{L} =
\oplus_{i,j}\mathbf{L}_{ij}$.

\begin{definition}\label{def:multidegree}
  Two tuples, $\mathbf{t} = (t_1, \ldots, t_m)$ and $\mathbf{s} =
  (s_1, \ldots, s_l)$ of non-negative integers form a system of orders
  for $\mathbf{L}$ if for each $1 \leq j \leq m$, $1 \leq i \leq l$ we
  have order $\mathbf{L}_{ij} \leq t_j - s_i$ (if $t_j - s_i < 0$ then
  $\mathbf{L}_{ij} = 0$). The $(\mathbf{t},\mathbf{s})$-principal part
  of $\mathbf{L}$ is obtained by replacing each $\mathbf{L}_{ij}$ by
  its terms which are exactly of order $t_j - s_i$, and the
  $(\mathbf{t},\mathbf{s})$-principal symbol of $\mathbf{L}$ is
  obtained by replacing each $\mathbf{L}_{ij}$ with its $t_j - s_i$
  principal symbol.
\end{definition}

We apply now this definition to our setup.

\begin{lemma} \label{lem:symbol} The leading symbol of $\mathbf{P}$ is
  given by the formula
  \begin{equation}
    \label{eq:symbP}
    \sigma_\mathbf{P}(v)(V,r)=(v\wedge \iota_{V} \Om,  v r, v\wedge\iota_V\om),
  \end{equation}
  where $v\in T^*\setminus M$. The tuples $\mathbf{t} = (2,2,2)$ and
  $\mathbf{s} = (1,1,0,1,1)$ form a system of orders for $\mathbf{L}$
  and the associated leading symbol is
  \begin{equation}
    \sigma_{\mathbf{L}}(v)(\dot \Om,\dot \theta,\dot\om)= (\sigma_{\mathbf{L}_1}(v),\sigma_{\mathbf{L}_2}(v),\sigma_{\mathbf{L}_3}(v),\sigma_{\mathbf{L}_4}(v),\sigma_{\mathbf{L}_5}(v)),
  \end{equation}
  with
  \begin{equation}
    \label{eq:symbolS}
    \begin{array}{ccc}
      \sigma_{\mathbf{L}_1}(v)(\dot \Om,\dot \theta,\dot\om) & = & v\wedge  \dot \Om,\\
      & & \\
      \sigma_{\mathbf{L}_2}(v)(\dot \Om,\dot \theta,\dot\om) & = & (v\wedge \dot \theta)^{(0,2)}, \\
      & & \\
      \sigma_{\mathbf{L}_3}(v)(\dot \Om,\dot \theta,\dot\om) & = &  v\wedge J (v\wedge \dot \om),\\
      & & \\
      \sigma_{\mathbf{L}_4}(v)(\dot \Om,\dot \theta,\dot\om) & = & v \wedge\left(2\vert\vert\Om\vert\vert_{\om}\dot\om\wedge\om+\delta(\vert\vert \Om\vert\vert_\om)\om^2\right), \\
      & & \\
      \sigma_{\mathbf{L}_5}(v)(\dot \Om,\dot \theta,\dot\om) & = & v \wedge \dot \theta\wedge \om^2.
    \end{array}
  \end{equation}
\end{lemma}
\begin{proof}
  To apply Definition \ref{def:multidegree} we use the direct sum
  decomposition \eqref{eq:T0Psplit}, setting $\Omega^0(E_1) =
  \Om^{3,0}\oplus\Om^{2,1}$, $\Omega^0(E_2) = \Om^1(i\RR)$ and
  $\Omega^0(E_3) = \Om^{1,1}_\RR$. The proof is a routine check and is
  left to the reader.
\end{proof}

A linear multi-degree complex of differential operators is elliptic if
the induced sequence of symbols is exact, as in the standard case. The
usual Fredholm properties of elliptic complexes hold, and therefore
given any elliptic complex we have an associated finite dimensional
cohomology.

\begin{proposition}
  \label{prop:elliptic}
  The sequence (\ref{eq:abcomplex}) is an elliptic complex. The space of infinitesimal deformations of the system
  \eqref{eq:stromabelian} is defined as the finite-dimensional vector space
$$
H^1(A^*) = \frac{\operatorname{Ker} \mathbf{L}}{\operatorname{Im}
  \mathbf{P}}.
$$
\end{proposition}

\begin{proof}
  By $\cX$-invariance of \eqref{eq:stromabelian}, $\mathbf{L} \circ
  \mathbf{P} = 0$. We prove next that the associated sequence of
  symbols is exact.  Assume that $\sigma_\mathbf{L}(v)(\dot \Om,\dot
  \theta,\dot\om)=0$ for $v \in T^*\setminus M$.  From the equations
  $\sigma_{\mathbf{L}_j}(v)=0$, $j=2, 5$, we deduce that there is a
  purely imaginary constant $r$ such that $\dot \theta=v r$.  Using
  the ellipticity of the complex \eqref{eq:CYinfinitesimal} and the
  isomorphism \eqref{eq:T0isom} there exists a unique $V\in T$ such
  that
$$
\dot \Om=v\wedge \iota_{V} \Om.
$$
In terms of $\dot J$, this translates to
\begin{equation}
  \label{eq:dotJtensor}
  \dot J= Jv\otimes V +v\otimes JV.
\end{equation}
It remains to show that $\dot\om=v\wedge \iota_V\om$. From
$\sigma_{\mathbf{L}_4}(v)=0$ we deduce
\begin{equation}\label{eq:sigma4}
  (J v) \wedge v \wedge\left(2\vert\vert\Om\vert\vert_{\om}\dot\om\wedge\om+\delta(\vert\vert \Om\vert\vert_\om)\om^2\right) = 0.
\end{equation}
Using now $J (\sigma_{\mathbf{L}_3}(v)) =
0$, 
we obtain
$$
\delta(\vert\vert \Om\vert\vert_\om)(J v) \wedge v \wedge \om^2 = 0,
$$
and from this,
\begin{equation}\label{eq:delomeganul}
  \delta(\vert\vert \Om\vert\vert_\om) = 0.
\end{equation}
Using now \eqref{eq:norm3form} and \eqref{eq:dotnormOmega}
\begin{equation}\label{eq:iVomombar}
  \begin{split}
    \|\Omega\|^2_{\omega} \dot \omega \wedge \omega^2 & = 2i \(\dot \Omega \wedge \overline{\Omega} + \Omega \wedge \overline{\dot \Omega}\)\\
    & = 2i v \wedge i_V (\Omega \wedge \overline{\Omega})\\
    & = \|\Omega\|^2_{\omega} v \wedge (i_V \omega) \wedge \omega^2.
  \end{split}
\end{equation}
Define $\tau = \dot \omega - v \wedge (i_V \omega)$ and notice from
\eqref{eq:dotomega02} and \eqref{eq:dotJtensor} that $\tau$ is a real
$(1,1)$-form. Furthermore, from \eqref{eq:iVomombar} and
\eqref{eq:delomeganul}, combined with the vanishing of
$\sigma_{\mathbf{L}_3}(v)$ and $\sigma_{\mathbf{L}_4}(v)$, we deduce
\begin{equation}\label{eq:3conditions}
  \begin{split}
    \tau \wedge \omega^2 &= 0,\\
    \tau \wedge v \wedge \omega & = 0,\\
    \tau \wedge v \wedge J v & = 0.
  \end{split}
\end{equation}
From the last equation
$$
\tau = v \wedge a + Jv \wedge b
$$
for suitable $1$-forms $a$ and $b$. Complete the family $\lbrace v, Jv
\rbrace$ into a basis of $T^*$ with forms $\lbrace X,JX,Y,JY \rbrace$
such that $\om$ is written
\begin{equation}
  \label{eq:simplifyom}
  \om=l v\wedge Jv + X\wedge JX + Y\wedge J Y
\end{equation}
for some real constant $l$. Now, from the second equation in
\eqref{eq:3conditions} we obtain that $b$ is proportional to $v$ and
therefore
$$
\tau = v \wedge (t_1 J v + t_2 X + t_3 JX + t_4 Y + t_5 JY)
$$
for suitable real constants $t_1, \ldots, t_5 \in \RR$.  We note that
a basis of the space of real $(1,1)$-forms is
\begin{align*}
  \Lambda^{1,1}_\RR= & \langle\: v\wedge Jv,\:X\wedge JX,\:Y\wedge JY, \: X\wedge v + JX\wedge J v, \\
  & J X\wedge v + Jv\wedge X, \: Y\wedge v + JY\wedge J v,\: J Y\wedge v + Jv\wedge Y,\\
  & X\wedge Y + JX\wedge J Y ,\: J X\wedge Y + JY\wedge X\:\rangle
\end{align*}
and therefore necessarily
$$
\tau = t_1 v \wedge J v.
$$
Finally, from the first equation in \eqref{eq:3conditions} we obtain
$t_1 = 0$, which implies that $\dot \omega = v \wedge (i_V \omega)$ as
claimed.
\end{proof}

\subsection{Extension of the complex}
\label{sec:extensioncomplex}
We extend the complex (\ref{eq:abcomplex}) in order to define a space
of obstructions to integrability of infinitesimal deformations of the
abelian equations \eqref{eq:stromabelian}.

We first define a complex parameterizing joint deformations of a
Calabi-Yau structure on $M$ endowed with a holomorphic line
bundle. For this, we combine an elliptic complex defined by Goto
\cite{Got} with previous work of Huang \cite{hu}. Goto's complex is an
extension of the complex (\ref{eq:CYinfinitesimal}) given by
\begin{equation}
  \label{eq:CYcomplex}
  \begin{array}{lll}
    0\rightarrow \Omega^{2,0} \lra{d} \Om^{3,0}\oplus \Om^{2,1}\lra{d} \Om^{3,1}\oplus\Om^{2,2}
    \lra{d} \Om^{3,2}\oplus\Om^{2,3}\lra{d} \Om^{3,3}\lra{} 0,
  \end{array}
\end{equation}
where we use the identification $\Om^0(T^{1,0}) \cong \Omega^{2,0}$
provided by the isomorphism \eqref{eq:T0isom}.  Following \cite{hu}
(cf. \cite{GFT}) we define an elliptic complex
\begin{equation}
  \label{eq:jointComplex}
  (C^*) \qquad \qquad 0\rightarrow C^0 \lra{\dbar_0} C^1 \lra{\dbar_0} C^2 
  \lra{\dbar_0} C^3 \lra{\dbar_0} C^4 \lra{} 0
\end{equation}
where
\begin{align*}
  C^j & := \Omega^{3,j-1}\oplus \Omega^{2,j} \oplus \Omega^{0,j},
  \qquad j \geq 0,
\end{align*}
(by convention $\Omega^{3,-1} = 0$) with differential given by
$$
\dbar_0(\alpha,\beta) = (d \alpha, \dbar \beta -
F^{T_0^{-1}(\alpha^{2,j})}).
$$

To include deformations of the metric, we build on the complex for the
Hermite--Yang--Mills equations defined by Kim \cite{kim} (see also
\cite[p. 246]{Kob}). Consider the following commutative diagram
\begin{equation}
  \label{eq:AComplex}
  \xymatrix{
    (\tilde{A^*}) \qquad \qquad 0 \ar[r] & \tilde{A^0} \ar[r]^{\mathbf{P}} \ar[d]^{p_0} & \tilde{A^1} \ar[r]^{\widetilde{\mathbf{L}}} \ar[d]^{p_1} & \tilde{A^2} \ar[r]^{\dbar_0 \oplus \widetilde{\mathbf{d}}} \ar[d]^{p_2} & \tilde{A^3} \ar[r]^{\dbar_0 \oplus d} \ar[d]^{p_3} & \tilde{A^4} \ar[r] \ar[d]^{p_4} & 0\\
    (C^*) \qquad \qquad 0 \ar[r] & C^0 \ar[r]^{\dbar_0} & C^1 \ar[r]^{\dbar_0} & C^2 \ar[r]^{\dbar_0} & C^3 \ar[r]^{\dbar_0} & C^4 \ar[r] & 0
  }
\end{equation}
with
\begin{align*}
  \tilde{A^0} &:= 
  A^0, & & p_0(V,r)  = (i_V\Omega,r), \\
  \tilde{A^1} &:= A^1, & &p_1(\dot \Omega, \dot \theta, \dot \omega)  = (\dot\Omega,\dot \theta^{0,1}), \\
  \tilde{A^2} &:= C^2 \oplus \Omega^{2,2}_\RR \oplus \Omega^5 \oplus \Omega^6(i\RR), & & p_2(s,\gamma,\delta,\epsilon)  = s, \\
  \tilde{A^j} & := C^j \oplus \Omega^{j+2} \oplus \Omega^{j+3}, & &
  p_j(s,\tau,\sigma) = s, \qquad j= 3,4,
\end{align*}
where $s$ denotes an element in $C^j$ for $j \geq 2$. It remains to
define the maps $\widetilde{\mathbf{L}}$ and $\widetilde{\mathbf{d}}$
in \eqref{eq:AComplex}, given by suitable modifications of the
operator $\mathbf{L}$ in \eqref{eq:abcomplex} (see Lemma
\ref{lem:diff}) and the exterior differential. Firstly, we define
$$
\widetilde{\mathbf{L}} = \mathbf{L}_1 \oplus \mathbf{L}_2 \oplus
\mathbf{L}_3^{2,2} \oplus \mathbf{L}_4 \oplus \mathbf{L}_5,
$$
where $\mathbf{L}_3^{2,2}$ denotes the $(2,2)$ part of $\mathbf{L}_3$
in (\ref{eq:diff}). We note that $d \circ \mathbf{L}_3 = 0$, but $d
\circ \mathbf{L}_3^{2,2}$ does not vanish in general. In fact, we have
the following formula.

\begin{lemma}\label{lem:S13}
  The $(1,3)$ part of $\mathbf{L}$ is given by \begin{equation}
    \label{eq:S13}
    \mathbf{L}_3^{1,3} =  -2i\del (\om^{T_0^{-1}(\mathbf{L}_1^{2,2})})- 2i(\partial \omega)^{T_0^{-1}(\mathbf{L}_1^{2,2})} - 2c \mathbf{L}_2 \wedge F.
  \end{equation}
\end{lemma}

This motivates the introduction of the map
$$
\widetilde{\mathbf{d}} \colon \tilde{A^2} \to \Omega^{5} \oplus
\Omega^{6},
$$
defined by
$$
\widetilde{\mathbf{d}}(\alpha,\beta,\gamma,\delta,\epsilon) = \(d
\(\gamma- 2\operatorname{Re}\( 2i(\partial
\omega)^{T_0^{-1}(\alpha^{2,2})} + 2c\beta\wedge F\)\) ,d\delta\)
$$

We are now ready to prove the main result of this section.
\begin{proposition}
  \label{prop:fullabeliancomplex}
  The sequence $\tilde{A^*}$ defines an elliptic complex of
  differential operators, whose first cohomology $H^1(\tilde{A^*})$
  equals the cohomology $H^1(A^*)$ of \eqref{eq:abcomplex}. The space
  of obstructions for the system \eqref{eq:stromabelian} is defined as
  $H^2(\tilde{A^*})$.
\end{proposition}

\begin{remark}
  The complex $\tilde{A^*}$ has a slightly different flavour from the
  one of Kim \cite{kim} (see also \cite{Kob}), due to the conformally
  balanced equation. In Kim's complex, the linearization of the
  hermitian-Yang-Mills equation is extended by zero, while in
  \eqref{eq:AComplex} we need to introduce an exterior differential to
  extend the linearization of the conformally balanced equation as
  part of an elliptic complex.
\end{remark}

We start with the proof of \eqref{eq:S13}. We need the following.

\begin{lemma}
  Let $\gamma \in \Om^d(T^{1,0})$ and $\alpha\in \Om^{p,q}$.  Then
  \begin{equation}
    \label{eq:delbagamma}
    \delb (\alpha^\gamma)= \alpha^{\delb \gamma} - (-1)^d (\delb \alpha )^\gamma.
  \end{equation}
\end{lemma}

\begin{proof}
  Write locally $\gamma= \sum_k \gamma_k \frac{\del}{\del z_k}$ so
  that
 $$
 \alpha^\gamma= \sum_k \gamma_k\wedge \iota_{\frac{\del}{\del
     z_k}}\alpha.
 $$
 Then
 $$
 \delb \alpha^\gamma= \sum_k (\delb \gamma_k)\wedge
 \iota_{\frac{\del}{\del z_k}}\alpha+ (-1)^d \sum_k \gamma_k\wedge
 \delb( \iota_{\frac{\del}{\del z_k}}\alpha).
 $$
 The formula
 $$
 \delb ( \iota_{\frac{\del}{\del z_k}}\alpha) = -
 \iota_{\frac{\del}{\del z_k}} \delb\alpha
 $$
 follows from $\iota_{\frac{\del}{\del z_k}} d\overline{z}_j=0$ and
 the expression of $\alpha$ in local coordinates.
\end{proof}

\begin{proof}[Proof of Lemma \ref{lem:S13}]
  By (\ref{eq:diff}) and \eqref{eq:dotomega02}
  \begin{equation}\label{eq:S5, part13}
    \begin{split}
      (\mathbf{L}_3(\dot\Om,\dot \theta, \dot \om))^{(3,1)+(1,3)}  = {}& -\partial \dbar(\omega^{\dot J}) - (d J (d\om)^{\dot J J})^{(3,1)+(1,3)}\\
      & - 2c (\partial \dot \theta^{1,0} + \dbar \dot
      \theta^{0,1})\wedge F,
    \end{split}
  \end{equation}
  and
  \begin{equation}
    (d J (d\om)^{\dot J J})^{1,3}  = \dbar((\partial \omega)^{\dot J^{1,0}}) - \partial((\dbar \omega)^{\dot J^{1,0}}).
  \end{equation}
  From formula (\ref{eq:delbagamma}), we obtain
  \begin{align*}
    (\mathbf{L}_3(\dot\Om,\dot \theta, \dot \om))^{(1,3)} & {} = -\del\delb (\om^{\dot J^{1,0}}) -\dbar((\partial \omega)^{\dot J^{1,0}}) + \partial((\dbar \omega)^{\dot J^{1,0}})\\
    & \phantom{{}=} - 2c (\dbar \dot \theta^{0,1})\wedge F \\
    & {} =  -\del (\om^{\dbar\dot J^{1,0}})-(\partial \omega)^{\delb\dot J^{1,0}}+(\del\delb \om)^{\dot J^{1,0}} \\
    & \phantom{{} =} - 2c (\dbar \dot \theta^{0,1})\wedge F.
  \end{align*}
  The proof follows now using the third equation in
  \eqref{eq:stromabelian}.
\end{proof}

Using Lemma \ref{lem:S13}, we go now for the proof of Proposition
\ref{prop:fullabeliancomplex}.

\begin{proof}[Proof of Proposition \ref{prop:fullabeliancomplex}]
  We first prove that \eqref{eq:AComplex} is a complex. Note that
  $\mathbf{L} \circ \mathbf{P} = 0$ implies $\widetilde{\mathbf{L}}
  \circ \mathbf{P} = 0$ and trivially $d\circ
  \widetilde{\mathbf{d}}=0$. Hence, using that $C^*$ is a complex, it
  remains to prove that $\widetilde{\mathbf{d}} \circ
  \widetilde{\mathbf{L}} = 0$. Given $(\dot\Om,\dot \theta, \dot
  \om)\in \tilde{A^1}$, using Lemma \ref{lem:S13} we obtain
$$
\widetilde{\mathbf{d}} \circ \widetilde{\mathbf{L}} (\dot\Om,\dot
\theta, \dot \om) = (d(2\operatorname{Re} (\partial(\omega^{\dbar \dot
  J^{1,0}}))),0),
$$
which vanishes for type reasons. We next prove that $\tilde{A^*}$ is
elliptic. Ellipticity at steps one, two and five follows from
ellipticity of the complex $C^*$ and the proof of Proposition
\ref{prop:elliptic}.  Ellipticity at step four follows from ellipticity
of $C^*$ and of the complex
$$\Om^{2,2}\lra{d} \Om^{5}\lra{d} \Om^{6}\lra{}0,$$
combined with the formula for the symbol of
$\widetilde{\mathbf{d}}$. We verified the ellipticity of $\tilde{A^*}$
at all steps but one, so it remains to show that an alternated sum of
dimensions vanishes. Given a complex $B^*$ as above and $x\in M$, set
$B_x^j$ to be the fiber at $x$ of the bundle on $M$ whose space of
smooth global sections is $B^j$. With this notation, we need to show
that
\begin{equation}
  \label{eq:alternatesumStrom}
  \sum_j (-1)^j \dim(\tilde{A^j_x})=0.
\end{equation}
By (\ref{eq:T0Psplit}),
\begin{equation}
  \label{eq:Tangentspace}
  \tilde{A^1} \cong C^1\oplus \Om^{1,1}_\RR
\end{equation}
and note that
\begin{equation}
  \label{eq:A0C0}
  \tilde{A^0}\oplus \Om^0\cong C^0.
\end{equation}
By ellipticity of $C^*$, the following sum vanishes:
\begin{equation}
  \label{eq:alternatesumC}
  \sum_j (-1)^j \dim(C_x^j)=0.
\end{equation}
Then, by (\ref{eq:alternatesumC}), (\ref{eq:Tangentspace}) and
(\ref{eq:A0C0}), equation (\ref{eq:alternatesumStrom}) is equivalent
to
\begin{align*}
  \dim (\Lambda^0T_x^*)  = {} & -\dim(\Lambda^{1,1}_\RR T_x^*)+\dim(\Lambda^{2,2}_\RR T_x^*\oplus\Lambda^5T_x^*\oplus \Lambda^6T_x^*)\\
  & -\dim(\Lambda^5T_x^*\oplus\Lambda^6T_x^*)+\dim(\Lambda^6T_x^*),
\end{align*}
which is trivially satisfied. Finally, by (\ref{eq:S13}),
$\widetilde{\mathbf{L}} =0$ is equivalent to $\mathbf{L}=0$, so the
first cohomology of $\tilde{A^*}$ equals the cohomology of
\eqref{eq:abcomplex}.
\end{proof}

\section{Infinitesimal moduli: general
  case}\label{sec:infmoduligeneral}

In this section we construct the infinitesimal moduli space of
solutions for the Strominger system. The main novelty with respect to
the abelian setting is that the symmetries of the system must preserve
the compatibility between the connection $\nabla$ and the
$\U(3)$-structure on the tangent bundle. As a result, the specific
gauge transformations that we consider do not have a group structure,
but nevertheless fit into the more general framework of Lie
groupoid-actions.

\subsection{Notation and parameter space}\label{sec:parameters}

Let $M$ be a compact, oriented, six dimensional smooth manifold.  Let
$P_{\GL^+}$ be the principal $\GL^+(6,\mathbb{R})$-bundle of oriented
frames of $M$. Let $P_K$ be a principal bundle, with compact structure
group $K$. Let $P$ be the principal $G$-bundle given by the fiber
product
$$
P = P_K \times_M P_{\GL^+}
$$
with $G = K \times \GL^+(6,\mathbb{R})$.  We fix a non-degenerate
pairing on the Lie algebra
$$
\mathfrak{g} = \mathfrak{k} \oplus \mathfrak{gl}(6,\RR)
$$ 
of $G$, given by
\begin{equation}\label{eq:pairingc}
  c = 2\alpha'(- \tr_\mathfrak{k} -  c_{\mathfrak{gl}}).
\end{equation}
Here, $- \tr_\mathfrak{k}$ denotes the Killing form on $\mathfrak{k}$
and $c_{\mathfrak{gl}}$ is a non-degenerate invariant metric on
$\mathfrak{gl}(6,\RR)$, which extends the non-degenerate Killing form
$-\tr$ on $\mathfrak{sl}(6,\RR) \subset \mathfrak{gl}(6,\RR)$.

Let $\cA$ denote the space of product connections $\theta = A \times
\nabla$ on $P$, where $A$ is a connection on $P_K$ and $\nabla$ is a
connection on the tangent bundle $T$ of $M$. As in Section
\ref{sec:infabel}, we denote by $\Omega^3_0 \subset \Omega^3_\CC$ the
space of complex $3$-forms $\Omega$ such that \eqref{eq:T01}
determines an almost complex structure $J_\Omega$ on $M$. For our
analysis, we consider a parameter space
$$
\cP \subset \Omega_0^3 \times \cA \times \Omega^2,
$$
defined by
$$
\cP=\lbrace (\Om,\theta,\omega)\;\vert\; \om \text{ is }
J_\Om-\text{compatible} \text{ and } \nabla \text{ is }
(J_\Om,\om)-\text{unitary} \rbrace.
$$
The points in $\cP$ are regarded as unknowns for the system of
equations
\begin{equation}\label{eq:stromnonabelian}
  \begin{split}
    d\Om & =0, \ \ \ \ \ \ \ \ \ \ \ \ \ \ d(\vert\vert \Om \vert \vert_\om \om^2)  =  0,\\
    F_\theta^{0,2} & = 0, \ \ \ \ \ \ \ \ \ \ \ \ \ \ \ \ \ \ \ F_\theta\wedge \omega^2 = 0  ,\\
    dd^c \om- c(F_\theta\wedge F_\theta) & = 0,
  \end{split}
\end{equation}
where $F_\theta$ denotes the curvature of $\theta = A \times \nabla$,
given explicitly by
$$
F_\theta = F_A + R_\nabla \in \Omega^2(\ad P).
$$
Our convention for the curvature tensor is \cite{BerlineGetzler}
$$
F_\theta = - \theta[\theta^\perp \cdot , \theta^\perp \cdot] \in
\Omega^2(\ad P),
$$
where $\theta$ is identified with a bundle morphism $\theta \colon TP
\to VP$, $\theta^\perp := \Id - \theta$ is the projection into the
horizontal subspace and the Bracket denotes the Lie bracket of vector
fields on $P$. Alternatively, we will use the formula
$$
F_\theta = d \theta + \frac{1}{2}[\theta,\theta],
$$
where $\theta$ is regarded as a $G$-invariant $1$-form in $P$ with
values in $\mathfrak{g}$ and the bracket is the one on the Lie
algebra. The induced covariant derivative on the bundle of Lie
algebras $\ad P = P \times_G \mathfrak{g}$ is
$$
i_V d^\theta r = [\theta^\perp V,r],
$$
which satisfies $d^\theta \circ d^\theta = [F_\theta,\cdot]$.

\begin{remark}
  The compatibility between $\nabla$ and the $\SU(3)$-structure on $M$
  is unmotivated from the physics point of view, but it is helpful for
  the mathematics that follow, making our discussion more
  standard. For the physics of the Strominger system it is required
  the weaker assumption that $\nabla$ is compatible with the metric
  underlying the $\SU(3)$-structure.
\end{remark}

Solutions of \eqref{eq:stromnonabelian} are in correspondence with
structures of Calabi-Yau manifold on $M$, endowed with a solution of
the Strominger system \eqref{eq:stromintro}, as it follows from
$$
c(F_\theta\wedge F_\theta) = \alpha'\(\tr R_\nabla \wedge R_\nabla -
\tr_\mathfrak{k} F_A \wedge F_A\).
$$
Note here that the compatibility between $\nabla$ and
$(J_\Omega,\omega)$ reduces the holonomy of $\nabla$ to $\U(3) \subset
\SL(6,\mathbb{R})$, and $c_{|\mathfrak{sl}(6,\mathbb{R})} = - \tr$.

We now proceed to the description of the tangent space of $\cP$. Fix
an element $(\Omega,\theta,\omega)$ of $\cP$ and set $T_0\cP$ to be
the tangent space of $\cP$ at this point.  Differentiating the
compatibility conditions $\nabla J = 0$ and $\nabla \omega = 0$, we
obtain
\begin{equation}
  \label{eq:compatibilityconditionsnonabelian}
  \nabla \dot J + [\dot \nabla,J] = 0, \qquad \nabla \dot \omega - \omega^{\dot \nabla} = 0,
\end{equation}
where $\dot \nabla \in \Omega^1(\End T)$ is the variation of $\nabla$ and
$$
\om^{\dot \nabla} (X,Y)= \om(\dot \nabla X, Y) + \om(X, \dot \nabla Y)
$$
for all $X,Y \in \Om^0(T)$. Then, the non skew-hermitian part of $\dot \nabla$ with respect to
$(J_\Omega,\omega)$ is determined by the variations $\dot J$ and $\dot
\omega$ and we find an isomorphism
\begin{equation}
  \label{eq:T0isomnonabelian}
  T_0\cP  \cong S_1 := \Omega^{3,0} \oplus \Omega^{2,1} \oplus \Omega^1(\ad P_h) \oplus \Omega^{1,1}_\RR,
\end{equation}
where
$$
P_h = P_K \times_M P_{\U(3)},
$$
and $P_{\U(3)} \subset P_{\GL^+}$ is the =$\U(3)$-reduction determined by
the bundle of unitary frames of $(J_\Omega,\omega)$. This isomorphism
is explicitly defined combining \eqref{eq:isoomega} and
\begin{equation}\label{eq:isonabla}
  \dot \nabla = \dot \nabla_h - \nabla \( \frac{1}{2}J \dot J + \frac{1}{4}(\dot \omega J + J\dot \omega) \),
\end{equation}
for $\dot \nabla_h \in \Omega^1(\ad P_{\U(3)})$.  Note that in
(\ref{eq:isonabla}), $\dot \om$ stands for the element of
$\Om^0(\End(T))$ associated to $\dot \om$ via the metric
$g=\om(\cdot,J\cdot)$. We also note that, unlike \eqref{eq:T0Psplit},
the isomorphism $S_1 \to T_0 \cP$ is a differential operator of order
$1$, and hence it is not induced by a bundle isomorphism. This will be
important in what follows.

\subsection{Gauge groupoid and infinitesimal action}
\label{sec:groupoid}
We define now the symmetries that we will use to construct the space
of infinitesimal variations of the Strominger system. Given $g \in
\Aut P$, we denote by $\check g \in \Diff$ the diffeomorphism in the
base that it covers. Consider the groupoid
$$
\ctG\rightrightarrows \cP
$$
given by
$$
\ctG : = \{(g,\Omega,\theta,\omega) \in \Aut P\times \cP\;\vert\;
g^*(J_\Omega,\omega) = \check{g}^*(J_\Omega,\omega)\},
$$
with source and target maps
$$
s (g,x) = x, \qquad t (g,x) = xg,
$$
where $x \in \cP$ and the right action is given by pull-back. We refer
to \cite{mac} for basic definitions on groupoids. The proof of the
following result is straightforward and is therefore omitted.

\begin{lemma}
  \label{lem:groupoidpreservessolutions}
  The action of $\ctG$ on $\cP$ preserves \eqref{eq:stromnonabelian},
  that is, if $x \in \cP$ is a solution of \eqref{eq:stromnonabelian}
  then $xg$ is also a solution for all $(g,x) \in s^{-1}(x)$.
\end{lemma}

We fix an element $(\Omega,\theta,\omega) \in \cP$ and denote $ F =
F_\theta$, $R=R_\nabla$ and $J = J_\Omega$. There is a natural
bijection
\begin{equation}\label{eq:groupbundlefibre}
  s^{-1}(\Omega,\theta,\omega) \cong \Aut P_K \times \cG_{\U(3)},
\end{equation}
where $\cG_{\U(3)}$ denotes the gauge group of the $\U(3)$-structure
$(J,\omega)$. Hence, using the connection $A$ on $P_K$, the vector
space of infinitesimal symmetries is given by
$$
S^0 = \Omega^0(T) \oplus \Omega^0(\ad P_h).
$$

\begin{lemma}\label{lem:infactionnabla}
  The infinitesimal action of $(V,t) \in \Omega^0(T) \oplus
  \Omega^0(\ad P_{\U(3)})$ on $\nabla$ is given by
  \begin{equation}\label{eq:infactionnabla}
    d^\nabla t + d^\nabla(\nabla V) + i_V R
  \end{equation}
  where $d^\nabla(\nabla V)$ denotes the exterior derivative
  $d^\nabla$ on $\Omega^*(\End T)$, induced by $\nabla$, acting on the
  endomorphism $\nabla V \in \Omega^0(\End T)$.
\end{lemma}
\begin{proof}
  For an arbitrary element $\zeta \in \Lie \Aut P_{\GL^+}$ we can write
  uniquely
$$
\zeta = t + \nabla^\perp V,
$$
for suitable $t \in \Omega^0(\ad P_{\GL^+})$ and $V \in \Omega^0(T)$,
where $\nabla^\perp V$ denotes the horizontal lift of $V$ with respect
to $\nabla$. Then, using this splitting, the infinitesimal action of
$\zeta$ on $\nabla$ reads
$$
\zeta \cdot \nabla = d^\nabla t + i_V R.
$$
Here we note that $\ad P_{\GL^+} = \End T$ and $\nabla t$ denotes the
covariant derivative on $\End T$ acting on the endomorphism $t \in
\Omega^0(\End T)$. Finally, given $V \in \Omega^0(T)$, we note that
$dV \colon T \to T(T)$ is the image of $V$ by the Lie algebra morphism
$$
\Omega^0(T) \to \Lie \Aut P_{\GL^+}
$$
induced by the natural inclusion $0 \to \Diff_0 \to \Aut P_{\GL^+}$,
and hence the result follows from the formula
$$
\nabla V = dV - \nabla^\perp V.
$$
\end{proof}

By Lemma \ref{lem:infactionnabla}, the infinitesimal action of $\ctG$
on $(\Omega,\theta,\omega) \in \cP$ reads:
\begin{equation}
  \label{eq:infinitesimalactionnonabelian0}
  (V,r) \cdot (\Omega,\theta,\omega) = (d\iota_{V} \Om, (r + \theta^\perp V)\cdot \theta,  \cL_V \om),
\end{equation}
where $r = s + t \in \Omega^0(\ad P_K \oplus \ad P_{\U(3)})$ and
\begin{equation}
  (r + \theta^\perp V)\cdot \theta = (d^A s + \iota_V F_A, d^\nabla t + d^\nabla(\nabla V) + i_V R).
\end{equation}
Here, $d^A$ denotes the exterior derivative induced by $A$ on
$\Om^*(\ad P_K)$. Using the previous formula and the isomorphism
\eqref{eq:T0isomnonabelian}, we define a differential operator
\begin{equation}
  \label{eq:infinitesimalactionnonabelian}
  \begin{array}{cccc}
    \mathbf{P} : & S^0 & \rightarrow &  S^1\\
    & (V,r) & \mapsto & (d\iota_{V} \Om, \dot \theta_h,  (\cL_V \om)^{1,1}),
  \end{array}
\end{equation}
where
$$
\dot \theta_h = (r + \theta^\perp V)\cdot \theta + \nabla \(
\frac{1}{2}J \dot J + \frac{1}{4}(\dot \omega J + J\dot \omega) \) \in
\Omega^1(\ad P_h),
$$
for $\dot \omega = \cL_V \omega$ and $\dot J$ given by
\eqref{eq:dotJtensor10}. We conclude this section with the following lemma, whose proof is left to the reader.
\begin{lemma}
  \label{lem:symbolPnonabelian}
  The tuples $\mathbf{t} = (2,1,1)$ and $\mathbf{s} = (1,0,0,1)$ form
  a system of orders for $\mathbf{P}$ and the associated leading
  symbol is given by
  \begin{equation}
    \label{eq:symbPT}
    \sigma_{\mathbf{P}}(v)(V,r)=(v\wedge \iota_{V} \Om,v\otimes s,  v \otimes (t + (v \otimes V)_h), (v\wedge\iota_V\om)^{1,1}),
  \end{equation}
  where $v \in T^*\setminus M$ and $(v \otimes V)_h$ is the
  skew-hermitian part of the endomorphism $v \otimes V$.
\end{lemma}

\subsection{Linearization and ellipticity}\label{sec:linnonabelian}

Let $(\Omega,\theta,\omega) \in \cP$ be a solution of
\eqref{eq:stromnonabelian}. Let us denote by $\mathbf{L}$ the
differential at $(\Omega,\theta,\omega)$ of \eqref{eq:stromnonabelian}
with respect to variations in $T_0\cP$. We set
$$S^2=\Omega^{3,1}\oplus \Omega^{2,2} \oplus \Omega^{0,2}(\ad P_h) \oplus \Omega^4 \oplus \Omega^5 \oplus \Omega^6(\ad P_h).$$
\begin{lemma}
  The operator $\mathbf{L}$ takes values in $S^2$ and induces an
  operator
  \begin{equation}
    \label{eq:diffnonabelian}
    \mathbf{L}=\oplus_{i=1}^5 \mathbf{L}_i: S^1 \rightarrow S^2
  \end{equation}
  via \eqref{eq:T0isomnonabelian}, given by
  \begin{equation}\label{eq:diffnonabelian2}
    \begin{split}
      \mathbf{L}_1(\dot \Om,\dot \theta_h,\dot\om^{1,1}) & = d\dot \Om,  \\
      \mathbf{L}_2(\dot \Om,\dot \theta_h,\dot\om^{1,1}) & = \dbar^\theta (\dot \theta_h)^{0,1} + \frac{i}{2}(F^{\dot J})^{0,2},\\
      \mathbf{L}_3(\dot \Om,\dot \theta_h,\dot\om^{1,1}) & = d\(J d \dot \om - J (d\om)^{\dot J J} - 2c(\dot \theta_h \wedge F)\),\\
      \mathbf{L}_4(\dot \Om,\dot \theta_h,\dot\om^{1,1})& = d\(2\vert\vert\Om\vert\vert_{\om}\dot\om\wedge\om+ \delta(\vert\vert \Om\vert\vert_\om) \om^2\),\\
      \mathbf{L}_5(\dot \Om,\dot \theta_h,\dot\om^{1,1}) & = d^\theta
      \dot \theta_h \wedge \om^2 + 2 F \wedge \dot\om\wedge \om,
    \end{split}
  \end{equation}
  where $\dot \omega = \dot \omega^{1,1} + \frac{1}{2}\dot
  \omega^{\dot J J}$, $\dot J$ is the infinitesimal variation of
  almost-complex structure \eqref{eq:dotJtensor10}, $\delta
  (\vert\vert \Om\vert\vert_\om) $ is given by formula
  (\ref{eq:dotnormOmega}), and $\delb^\theta=\pr^{(0,1)}\circ
  d^\theta$ is the Dolbeault operator induced by $\theta$.
\end{lemma}
\begin{proof}
  The computations follow as in the proof of Lemma \ref{lem:diff}. The
  crucial step is to show that the only contribution to $\mathbf{L}
  \colon T_0\cP \to S^2$ in the variation of $F$
$$
d^\theta \dot \theta = d^A \dot A + d^\nabla \dot \nabla
$$
comes from $d^A \dot A + d^\nabla \dot \nabla_h$, where $\dot \theta =
\dot A + \dot \nabla$ and $\dot \theta_h = \dot A + \dot \nabla_h$
(see \eqref{eq:isonabla}). To simplify the exposition, for a moment we
denote
\begin{equation}\label{eq:bJOm}
  q = \frac{1}{2}J \dot J + \frac{1}{4}(\dot \omega J + J\dot \omega)
\end{equation}
so that
$$
\dot \nabla = \dot \nabla_h - \nabla q.
$$
With this notation,
\begin{equation}
  \label{eq:dnablanabla}
  \begin{split}
    d^\nabla \dot \nabla & = d^\nabla \dot \nabla_h -d^\nabla d^\nabla q\\
    & =d^\nabla \dot \nabla_h - [R, q].
  \end{split}
\end{equation}
Then, as $(\Om,\theta,\om)$ is a solution of
\eqref{eq:stromnonabelian}, $R^{0,2}=0$ and $R\wedge \om^2=0$, and
therefore
\begin{align*}
  [R, q]^{0,2} & =0,\\
  [R,q]\wedge\om^2 & =0.
\end{align*}
Jointly with \eqref{eq:dnablanabla}, formulae for $\mathbf{L}_2$ and
$\mathbf{L}_5$ follow.  Finally, to calculate $\mathbf{L}_3$, we note
that
$$
\tr([R,q]\wedge R)=0.
$$
\end{proof}

Consider the following complex of differential operators
\begin{equation}
  \label{eq:nonabcomplex}
  (S^*) \qquad \qquad \qquad S^0 \lra{\mathbf{P}}  S^1  \lra{\mathbf{L}} S^2.
\end{equation}
\begin{theorem}
  \label{prop:ellipticnonabelian}
  The complex (\ref{eq:nonabcomplex}) is elliptic. The space
  of infinitesimal deformations of solutions to the Strominger system
  of equations is defined as the first cohomology group $H^1(S^*)$.
\end{theorem}

\begin{proof}
  As the $\ctG$-action preserves solutions of
  \eqref{eq:stromnonabelian}, $\mathbf{L}\circ \mathbf{P}=0$.  Assume
  that $\sigma_\mathbf{L}(v)(\dot\Om,\dot \theta_h,\dot \om^{1,1})=0$,
  with $\dot \theta_h = \dot A + \dot \nabla_h$ (see
  \eqref{eq:bJOm}). Then, arguing as in the proof of Proposition
  \ref{prop:elliptic} we obtain
$$
(\dot \Om,\dot A,\dot \nabla_h, \dot \om^{1,1})=(v\wedge \iota_{V}
\Om,v\otimes s,v\otimes t, (v\wedge\iota_V\om)^{1,1})
$$
for some $V \in T$ and $s + t \in \ad P_h$. Trivially,
$$
\dot \nabla_h = v \otimes (t' + (v \otimes V)_h)
$$
for $t' \in \ad P_h$ and hence the result follows.
\end{proof}

We finish this section extending the complex \eqref{eq:nonabcomplex},
in order to define a space of obstructions to integrability for
infinitesimal deformations of the Strominger system. We follow closely
the method in Section \ref{sec:extensioncomplex}. Define
\begin{align*}
  \tilde{S^0} &:= 
  S^0,\\
  \tilde{S^1} &:= S^1,\\
  \tilde{S^2} &:= \Omega^{3,1}\oplus \Omega^{2,2} \oplus \Omega^{0,2}(\ad P_h) \oplus \Omega^{2,2}_\RR \oplus \Omega^5 \oplus \Omega^6(\ad P_h),\\
  \tilde{S^3} & := \Omega^{3,2}\oplus \Omega^{2,3} \oplus \Omega^{0,3}(\ad P_h) \oplus \Omega^{5} \oplus \Omega^6,\\
  \tilde{S^4} & := \Omega^{3,3} \oplus \Omega^{6},
\end{align*}
and
\begin{equation}\label{eq:SComplex}
  \xymatrix{
    (\tilde S^*) \qquad \qquad 0 \ar[r] & \tilde S^0 \ar[r]^{\mathbf{P}} & \tilde S^1 \ar[r]^{\widetilde{\mathbf{L}}} & \tilde S^2 \ar[r]^{\dbar_0 \oplus \widetilde{\mathbf{d}}} & \tilde S^3 \ar[r]^{\dbar_0 \oplus d}  & \tilde S^4 \ar[r] & 0.
  }
\end{equation}
Here
$$
\dbar_0(\alpha,\beta) = (d \alpha, \dbar^\theta \beta -
F^{T_0^{-1}(\alpha^{2,j})}).
$$
for $\alpha \in \Omega^{3,j-1}\oplus \Omega^{2,j}$ and $\beta \in
\Omega^{0,j}(\ad P_h)$, for $j = 2,3$, and
$$
\widetilde{\mathbf{d}} (\alpha,\beta,\gamma,\delta,\epsilon)= \(d
\(\gamma - 4\operatorname{Re}\( i(\partial
\omega)^{T_0^{-1}(\alpha^{2,2})} + \frac{1}{2}c(\beta\wedge F)\)\),
d\delta\).
$$

The proof of the following result follows easily combining the proof
of Proposition \ref{prop:fullabeliancomplex} with the ellipticity of
Kim's complex \cite{kim} (see also \cite[p. 246]{Kob}).

\begin{theorem}
  \label{theo:ellipticnonabelian}
  The complex $\tilde S^*$ is elliptic and there is a natural
  isomorphism $H^1(\tilde S^*) = H^1(S^*)$. The space of
  obstructions to integrability of infinitesimal deformations of the
  Strominger system of equations is defined as the second cohomology group
  $H^2(\tilde S^*)$ of $\tilde S^*$.
\end{theorem}

\section{Anomaly cancellation, flux quantization and
  generalized~geometry}
\label{sec:anomalyflux}

The aim of this section is to initiate the study of the geometry of
the moduli space of solutions of the Strominger system which emerges
from the Bianchi identity. We describe the infinitesimal structure of
a natural foliation in the moduli space, whose leaves are intimately
related to generalized geometry.

\subsection{An integrable distribution in the moduli space}

This prelude intends to serve as motivation for the rest of this
section. Our discussion is to be taken rather formally, since we do
not want to get involved here with the differential-topological
aspects of the problem. In particular, the construction of a natural
differentiable structure on the moduli space of solutions of the
Strominger system 
will be addressed in a sequel of the present work.

Following the notation of Section \ref{sec:infmoduligeneral},
the 
moduli space of solutions of the Strominger system can be identified
with the quotient
$$
\mathcal{M} = \cP_{S} / \cX,
$$
where $\cP_S \subset \cP$ is the locus where the equations
\eqref{eq:stromnonabelian} are satisfied. Using the transgression
formula for the Chern--Simons $3$-form, given any point
$(\Omega,\theta,\omega)$ in $\mathcal{M}$ we define a
$H^3(M,\RR)$-valued function
$$
\vartheta(\Omega',\theta',\omega') = [d^{c'}\omega' - d^c\omega - 2c(a
\wedge F) - c(a \wedge d^\theta a) - \frac{1}{3}c(a \wedge [a,a])],
$$
where $\theta' = \theta + a$. This function is well-defined up to the
action of the equivariant mapping class group of the principal bundle
$P$ and hence it defines a global, closed, $H^3(M,\RR)$-valued
$1$-form
$$
\delta \in \Omega^1(\cM,H^3(M,\RR)),
$$
induced by the explicit expression
$$
\delta(\dot \Omega,\dot \theta, \dot \omega) = [J d \dot \om - J
(d\om)^{\dot J J} - 2c(\dot \theta_h \wedge F)].
$$
Neglecting obstructions to integrability, the tangent space of $\cM$
at a point can be identified with the cohomology group $H^1(S^*)$ and
the closed $1$-form $\delta$ induces now an integrable distribution
\begin{equation*}
  \xymatrix{
    0 \ar[r] & \operatorname{Ker} \delta \ar[r] & H^1(S^*) \ar[r]^{\delta} & H^3(M,\RR)\\
  }
\end{equation*}
and therefore a foliation on the moduli space. The leaves of this
foliation are closely related to generalized geometry, as we will
rigorously show in this section at the infinitesimal level.

Motivation for our construction comes from the \emph{Green-Schwarz
  mechanism} \cite{GreenSchwarz} and the \emph{flux quantization
  condition}, two basic principles in the physics of the heterotic
string. On the one hand, physics claims that the fundamental equation
relating the unitary structure on the manifold with the pair of
connections $\nabla$ and $A$ is a local equation for $3$-forms, the
so-called \emph{Green-Schwarz anomaly cancellation condition} (or
\emph{anomaly equation}, for short)
\begin{equation}\label{eq:anomaly}
  H = db - \alpha'(CS(\nabla) - CS(A)),
\end{equation}
written in terms of the Chern--Simons $3$-forms of $A$ and $\nabla$,
where $H := d^c\omega$ is the \emph{$3$-form flux}. This requires the
introduction of an additional ingredient: a local real $2$-form
$b$---usually known as \emph{B-field potential}. The Bianchi identity
\eqref{eq:bianchiintro} is obtained by taking the exterior derivative
in \eqref{eq:anomaly}.  On the other hand, the second principle says
that the `closed part' of the flux $H$, obtained by gluing the local
closed $3$-forms $db$, is quantized (see e.g. \cite{AGS,OssaSvanes}),
giving a fixed integral cohomology class in
$H^3(M,\ZZ)$. 
Although the Green-Schwarz mechanism can be rigorously understood in topological grounds using the theory of \emph{differential string structures} \cite{SSS}, here we shall focus on its interpretation in terms of geometry in the moduli space. In this setup, level sets of the multi-valued function $\vartheta$ correspond formally to solutions of the anomaly equation \eqref{eq:anomaly}, while the kernel of the $1$-form $\delta$ can be interpreted as infinitesimal variations which preserve the flux quantization.



\subsection{Constraining infinitesimal variations: abelian
  case}\label{sec:constraining}

To simplify the exposition and highlight the main ideas, we start our
discussion with the abelian case studied in Section
\ref{sec:infmoduliabelian}. 
Let $H^1(A^*)$ be the finite-dimensional vector space of infinitesimal
variations of the given solution $(\Omega,\theta,\om)$, constructed in Section
\ref{sec:infmoduliabelian}. We note from \eqref{eq:diff} that there is
a natural linear map
\begin{equation}\label{eq:fluxmapab}
  \delta \colon  H^1(A^*) \to H^3(M,\RR)
\end{equation}
given by
$$
[(\dot \Om,\dot \theta,\dot\om)] \mapsto [J d \dot \om - J
(d\om)^{\dot J J} - 2c (\dot \theta \wedge F)],
$$
and we denote by $H^1(\mathring{A}^*) : = \Ker \delta$. For an element
$[(\dot \Om,\dot \theta,\dot\om)]$ of $H^1(\mathring{A}^*)$, the
equation $\mathbf{L}_3((\dot \Om,\dot \theta,\dot\om)) = 0$ is
satisfied as a consequence of the stronger condition
\begin{equation}\label{eq:anomalyvariationab}
  J d \dot \om - J (d\om)^{\dot J J} - 2c (\dot \theta\wedge F) = db,
\end{equation}
for a $2$-form $b \in
\Omega^2$. 

Being a fundamental parameter space, we would like to identify
precisely the objects that $H^1(\mathring{A}^*)$ parameterizes. In a
first approximation, it is natural to ask whether
$H^1(\mathring{A}^*)$ arises naturally as the cohomology of a complex
of geometric origin. We give next an affirmative answer to this
preliminary question, but somehow in an unexpected way: to construct
the complex, we need to enlarge the first two vector spaces in
\eqref{eq:abcomplex} (or \eqref{eq:AComplex}) by adding the space of
$2$-forms $\Omega^2$. The $2$-forms become part of the symmetries of
the problem, but also contribute as additional data for the geometric
objects parameterized by $H^1(\mathring{A}^*)$.  Consider the
following sequence
\begin{equation}
  \label{eq:abcircComplex0a}
  A^0 \oplus \Omega^2 \lra{\mathring{\mathbf{P}}} A^1 \oplus \Omega^2 \lra{\mathring{\mathbf{L}}} \mathring{\mathcal{R}},
\end{equation}
where we follow the notation of Section
\ref{sec:extensioncomplex}. Here,
$$
\mathring{\mathcal{R}} := \Omega^{3,1} \oplus \Omega^{2,2} \oplus
\Omega^{0,2}(i\RR) \oplus \Omega^3 \oplus \Omega^5 \oplus
\Omega^6(i\mathbb{R}),
$$
and the maps $\mathring{\mathbf{P}}$ and $\mathring{\mathbf{L}}$ are
defined by
\begin{align*}
  \mathring{\mathbf{P}}(V,r,B) = (\mathbf{P}(V,r),-B)
\end{align*}
and
$$
\mathring{\mathbf{L}} = \mathbf{L}_1 \oplus \mathbf{L}_2 \oplus
\mathring{\mathbf{L}}_3 \oplus \mathbf{L}_4 \oplus \mathbf{L}_5,
$$
where
$$
\mathring{\mathbf{L}}_3(\dot \Om,\dot \theta,\dot\om,b) = J d \dot
\om - J (d\om)^{\dot J J} - 2c (\dot \theta\wedge F) - db.
$$
The main difference with \eqref{eq:abcomplex} is that we have
decreased by one the degree of the original map $\mathbf{L}_3$ in such
a way that $\mathring{\mathbf{L}}_3 = 0$ is precisely the equation
\eqref{eq:anomalyvariationab}. Unlike \eqref{eq:abcomplex}, equation
\eqref{eq:abcircComplex0a} is a sequence of differential operators of
degree $1$, but it is not a complex.

\begin{lemma}
  The sequence \eqref{eq:abcircComplex0a} defines a complex by
  restriction of the first arrow to elements $(V,r,B) \in A^0 \oplus
  \Omega^2$ satisfying the equation
  \begin{equation}\label{eq:complexconditionab}
    \cL_V (d^c\om) = 2c((dr + \iota_V F)\wedge F) - dB.
  \end{equation}
\end{lemma}

The proof follows from a direct calculation of
$\mathring{\mathbf{L}}_3 \circ \mathring{\mathbf{P}}$ and is therefore
omitted.  Note that using the \emph{Bianchi identity} (last equation
in \eqref{eq:stromabelian}), we can rewrite
\eqref{eq:complexconditionab} as
\begin{equation}\label{eq:complexconditionsimplab}
  d(i_V d^c\om - 2c r F) = - dB.
\end{equation}
A fact worth mentioning is that the space of $(V,r,B)$ satisfying
\eqref{eq:complexconditionab} has a natural Lie algebra structure. We
will come back to this in Section \ref{sec:GGstrom}. By now, we use
the suggestive notation
\begin{equation}\label{eq:LieAutEab}
  \operatorname{Lie} \mathring{\operatorname{Aut} E} := \{(V,r,B) \in A^0 \oplus \Omega^2 \;\; \textrm{satisfying} \; \; \eqref{eq:complexconditionab}\}.
\end{equation}

\begin{lemma}\label{lem:mathringA}
  The cohomology of the complex of vector spaces
  \begin{equation}
    \label{eq:abcircComplex}
    (\mathring{A}^*) \qquad \qquad \operatorname{Lie} \mathring{\operatorname{Aut} E} \lra{\mathring{\mathbf{P}}} A^1 \oplus \Omega^2 \lra{\mathring{\mathbf{L}}} \mathring{\mathcal{R}}
  \end{equation}
  is naturally isomorphic to $H^1(\mathring{A}^*)$. Consequently, the
  cohomology of $\mathring{A}^*$ is finite dimensional.
\end{lemma}
\begin{proof}
  For a moment, denote by $H^1$ the cohomology of
  \eqref{eq:abcircComplex}. Then, we have a natural linear map $\Phi
  \colon H^1 \to H^1(\mathring{A}^*)$ given by sending a class $[(\dot
  \Om,\dot \theta,\dot\om,b)]$ in $H^1$ to the class $[(\dot \Om,\dot
  \theta,\dot\om)]$ in $H^1(\mathring{A}^*)$. By definition of
  $H^1(\mathring{A}^*)$, this map is trivially surjective. Assume now
  that $[(\dot \Om,\dot \theta,\dot\om,b)] \in H^1$ is such that
  $[(\dot \Om,\dot \theta,\dot\om)] = 0$ in
  $H^1(\mathring{A}^*)$. Then, $(\dot \Om,\dot \theta,\dot\om) =
  \mathbf{P}(V,r)$ and from the condition
$$
\mathring{\mathbf{L}}_3(\mathbf{P}(V,r),b) = 0
$$
it follows that $(V,r,b) \in \operatorname{Lie}
\mathring{\operatorname{Aut} E}$, proving that the map $\Phi$ is
injective.
\end{proof}

Note that the complex \eqref{eq:abcircComplex} is not a complex of
differential operators in the standard sense. This is due to the fact
that $\operatorname{Lie} \mathring{\operatorname{Aut} E}$ is described
by the differential equation \eqref{eq:complexconditionab} and hence
it does not correspond to the space of sections of any vector
bundle. Therefore, the theory of elliptic operators does not apply to
\eqref{eq:abcircComplex}, and there is no direct way of proving that
its cohomology is finite dimensional. We obtain this result by
comparison with the elliptic complex \eqref{eq:abcomplex}.

\subsection{Weakening the symmetries: the elliptic complex $\widehat
  A^*$}
\label{sec:hatAstar}

Leaving aside for a moment the geometric interpretation of the
parameter space $H^1(\mathring{A}^*)$, we would like to construct a
related complex $\widehat A^*$
which 
is, unlike $\mathring{A}^*$, an elliptic complex of differential
operators in the standard sense. Furthermore, we will construct
$\widehat A^*$ using degree $1$ operators, and therefore simplifying
the complex $A^*$. The cohomology of this new complex, constructed by
weakening the symmetries of $\mathring{A}^*$, will fit into an exact
diagram of linear maps
\begin{equation*}
  \xymatrix{
    & 0  \ar[d]   &  &\\
    & H^2(M,\RR)  \ar[d]  &  &\\
    & H^1(\widehat{A}^*)  \ar[d] & & & \\
    0 \ar[r] & H^1(\mathring{A}^*) \ar[d] \ar[r] & H^1(A^*) \ar[r]^\delta & H^3(M,\RR). \\
    & 0 &  & 
  }
\end{equation*}
In the way, we will pin down some of the additional ingredients that
we need to describe $H^1(\mathring{A}^*)$.  Consider the vector bundle
$$
E = T \oplus i \RR \oplus T^*
$$
and note that $\Omega^0(E) := A^0 \oplus \Omega^1$. We define a
sequence of differential operators
\begin{equation}
  \label{eq:circComplexinnerab}
  (\widehat{A}^*) \qquad \qquad \Omega^0(E) \lra{\mathring{\mathbf{P}}} A^1 \oplus \Omega^2 \lra{\mathring{\mathbf{L}}} \mathring{\mathcal{R}},
\end{equation}
where the map $\mathring{\mathbf{P}}$ is now defined by
\begin{align*}
  \mathring{\mathbf{P}}(V,r,\xi) = (\mathbf{P}(V,r),d\xi +
  i_V(d^c\omega) - 2c r F).
\end{align*}
A straightforward calculation shows that

\begin{equation}\label{eq:innerouterab}
  (V,r, - d\xi - i_V(d^c\omega) + 2c r F) \in \operatorname{Lie} \mathring{\operatorname{Aut} E}
\end{equation}
and therefore \eqref{eq:circComplexinnerab} is a complex of
differential operators.

\begin{proposition}\label{prop:Scircexactseqab}
  The sequence \eqref{eq:circComplexinnerab} is an elliptic complex of
  differential operators of degree $1$. There is an exact sequence
  \begin{equation*}
    0 \to H^2(M,\RR) \to H^1( \widehat A^*) \to H^1(\mathring{A}^*) \to 0
  \end{equation*}
  where $H^1( \widehat A^*)$ denotes the cohomology of
  \eqref{eq:circComplexinnerab}.
\end{proposition}
\begin{proof}
  By ellipticity of \eqref{eq:abcomplex},
  $\sigma_{\mathring{\mathbf{L}}}(\dot \Om,\dot \theta,\dot\om,b) = 0$
  implies $(\dot \Om,\dot \theta,\dot\om) =
  \sigma_{\mathring{\mathbf{P}}}(V,r)$. In addition,
$$
0 = \sigma_{\mathring{\mathbf{L}}_3}(v)(\dot \Om,\dot
\theta,\dot\om,b) = J (v \wedge v \wedge i_V\omega) - v \wedge b = -
v \wedge b
$$
and therefore $b = v \wedge v'$ by ellipticity of the De Rham complex,
proving ellipticity of \eqref{eq:circComplexinnerab}.

The linear map $\Phi$ which sends a class $[(\dot \Om,\dot
\theta,\dot\om,b)] \in H^1( \widehat A^*)$ to the class of $(\dot
\Om,\dot \theta,\dot\om,b)$ in $H^1(\mathring{A}^*)$ is trivially
well-defined and surjective. Assume now that the class of $(\dot
\Om,\dot \theta,\dot\om,b)$ is zero in $H^1(\mathring{A}^*)$. Then,
$(\dot \Om,\dot \theta,\dot\om) = \mathbf{P}(V,r)$ and from the condition
$\mathring{\mathbf{L}}_3(\mathbf{P}(V,r),b) = 0$ and
\eqref{eq:complexconditionsimplab} it follows that
$$
d(b - i_V(d^c\omega) + 2c r F) = 0.
$$
We can construct a map $\Ker \Phi \to H^2(M,\RR)$, defined by
$$
[(\mathbf{P}(V,r),b)] \mapsto [b - i_V(d^c\omega) + 2c r F],
$$
which is an isomorphism.
\end{proof}

\begin{remark}
  Note the strong analogy between $\Omega^0(E)$ and
  $\operatorname{Lie} \mathring{\operatorname{Aut} E}$ and,
  respectively, the Hamiltonian vector fields and the symplectic
  vector fields on a symplectic manifold. In fact, the natural map
  defined by \eqref{eq:innerouterab} defines an exact sequence
$$
0 \to \Omega^0(E) \to \operatorname{Lie} \mathring{\operatorname{Aut}
  E} \to H^2(M,\RR) \to 0,
$$
which provides an analogue for the (differential of the) flux map in
symplectic geometry. In the next section, we will see that
$\operatorname{Lie} \mathring{\operatorname{Aut} E}$ has the structure
of a Lie algebra while $\Omega^0(E)$ can be endowed with the structure
of a \emph{Courant algebra}.
\end{remark}

\subsection{Relation with generalized geometry of type
  $B_n$}\label{sec:relationGG}

The construction in Section \ref{sec:constraining} and Section
\ref{sec:hatAstar} may look rather strange from a classical
perspective, but turns out to be very natural in generalized
geometry. The precise theoretical framework that we need was
introduced by Baraglia \cite{Bar} and developed by the second author
\cite{Rubioth}, and goes under the name of \emph{generalized geometry
  of type $B_n$}.

To explain the most basic aspects of this relation, in this section we
consider the vector bundle
\begin{equation}\label{eq:Edefab}
  E = T \oplus i \RR \oplus T^*
\end{equation}
with an additional structure---the one of a smooth Courant
algebroid. This bundle is such that its space of global sections
$\Omega^0(E)$ provides the infinitesimal symmetries for the elliptic
complex \eqref{eq:circComplexinnerab} and the Lie algebra of
infinitesimal automorphisms of the Courant algebroid structure
contains $\operatorname{Lie} \mathring{\operatorname{Aut} E}$ as a Lie
subalgebra.  In addition, the space of parameters $A^1 \oplus
\Omega^2$---the middle step in the complexes \eqref{eq:abcircComplex}
and \eqref{eq:circComplexinnerab}---will be naturally interpreted as a
space of infinitesimal variations of a generalized metric on $E$. The
explanation of more intricate aspects of our construction in terms of
generalized geometry is postponed to Section \ref{sec:stgen}.

We fix our solution $(\Omega,\theta,\omega)$ of the abelian system
\eqref{eq:stromabelian}.  We shall emphasize that the rest of this
section relies solely on having a solution of the abelian Bianchi
identity
\begin{equation}\label{eq:bianchisimplab}
  d H = c(F \wedge F)
\end{equation}
with $H = d^c \omega$ and does not use the other structures provided
by the equations. Consider the smooth vector bundle \eqref{eq:Edefab}
endowed with the symmetric pairing
$$
\langle X + r + \xi,Y + t + \eta\rangle = \frac{1}{2}(\eta(X) +
\xi(Y)) + c(rt),
$$
and the canonical projection
$$
\pi \colon E \to T.
$$
Using the quantities $H$ and $F$, we can endowed $\Omega^0(E)$ with a
Dorfman bracket
\begin{equation}\label{eq:bracketab}
  \begin{split}
    [X+r+\xi,Y+t+\eta]  = {} &[X,Y] + L_{X}\eta - i_{Y}d\xi + i_{Y}i_{X}H\\
    & - F(X,Y) + i_Xdt - i_Y dr\\
    & + 2c(tdr) + 2c(i_XF t) - 2c(i_YF r).
  \end{split}
\end{equation} 
It can be checked from the Bianchi identity \eqref{eq:bianchisimplab}
that $(E,\la\cdot,\cdot\ra,[\cdot,\cdot],\pi)$ satisfies the axioms of
a Courant algebroid (see \cite[Sec. 2.3.3]{Rubioth}).

\begin{definition}\label{def:Courant}
  A Courant algebroid $(E,\la\cdot,\cdot\ra,[\cdot,\cdot],\pi)$ over a
  manifold $M$ consists of a vector bundle $E\to M$ together with a
  non-degenerate symmetric bilinear form $\la\cdot,\cdot\ra$ on $E$, a
  (Dorfman) bracket $[\cdot,\cdot]$ on the sections $\Omega^0(E)$, and
  a bundle map $\pi:E\to TM$ such that the following properties are
  satisfied, for $e,e',e''\in \Omega^0(E)$ and $\phi\in \cCi(M)$:
  \begin{itemize}
  \item[(D1):] $[e,[e',e'']] = [[e,e'],e''] + [e',[e,e'']]$,
  \item[(D2):] $\pi([e,e'])=[\pi(e),\pi(e')]$,
  \item[(D3):] $[e,\phi e'] = \pi(e)(\phi) e' + \phi[e,e']$,
  \item[(D4):] $\pi(e)\la e', e'' \ra = \la [e,e'], e'' \ra + \la e',
    [e,e''] \ra$,
  \item[(D5):] $[e,e']+[e',e]=2\pi^* d\la e,e'\ra$.
  \end{itemize}
\end{definition}
Automorphisms of this object were characterized in
\cite[Prop. 2.23]{Rubioth}. For our discussion, we just need the
infinitesimal automorphisms, given by a linear subspace
$$
\operatorname{Lie} \operatorname{Aut} E \subset \Omega^0(T) \oplus
\Omega^1(i\RR) \oplus \Omega^2.
$$

\begin{proposition}[\cite{Rubioth}]
  The Lie algebra of infinitesimal automorphisms of $E$ is given by
  \begin{equation}\label{eq:LieAutEabrubio}
    \operatorname{Lie} \operatorname{Aut} E = \{(V,a,B) \; : \; \cL_V F = da, \cL_V H = - dB + 2 c(a \wedge F)\}.
  \end{equation}
\end{proposition}

From \eqref{eq:LieAutEabrubio} and \eqref{eq:LieAutEab}, we obtain a
natural map
$$
\operatorname{Lie} \mathring{\operatorname{Aut} E} \to
\operatorname{Lie} \operatorname{Aut} E \colon (V,r,B) \mapsto (V,dr +
i_VF,B).
$$
In fact, the image defines a subspace closed under the Lie bracket
(see Corollary \ref{cor:ringAutE}), which justifies the notation and
gives the desired interpretation of $\operatorname{Lie}
\mathring{\operatorname{Aut} E}$. The space of sections $\Omega^0(E)$
corresponds to the \emph{inner symmetries} of $E$, via the natural
action induced by the Dorfman bracket.

To interpret geometrically the space of parameters $A^1 \oplus
\Omega^2$ in \eqref{eq:abcircComplex} and
\eqref{eq:circComplexinnerab}, we need generalized metrics. A
generalized metric on $E$ is given by a linear subspace
$$
V_+ \subset E
$$
such that the restriction of the pairing on $E$ is non-degenerate. For
simplicity, we assume $c < 0$ and the induced metric on $V_+$ to be
positive definite. Then, a generalized metric is equivalent to a
Riemannian metric $g$ on $M$ and an isotropic splitting of $E$ that,
using the canonical isotropic splitting of \eqref{eq:Edefab}, can be
regarded as an orthogonal transformation by $(b,a) \in \Omega^2 \oplus
\Omega^1(\ad P)$. Hence, the fixed solution of the abelian system
\eqref{eq:stromabelian} determines a generalized metric $V_+$, such
that the metric is determined by the $\SU(3)$-structure
$(\Omega,\omega)$, $b = 0$ and $a = 0$. Variations of $V_+$ are given
by variations of the metric, the imaginary $1$-form $a$ and the
$2$-form $b \in \Omega^2$, which are precisely the elements parameterized by $A^1 \oplus \Omega^2$.

The conclusion is that the finite-dimensional vector space
$H^1(\mathring{A}^*)$ corresponds to a space of infinitesimal
variations of $V_+$ as a generalized metric modulo the natural
symmetries of the smooth Courant algebroid $E$ \eqref{eq:Edefab},
while $H^1(\widehat{A}^*)$ is cut out by \emph{inner symmetries} of
$E$. Of course, $H^1(\mathring{A}^*)$ and $H^1(\widehat{A}^*)$ contain
more information related to variations of the $\SU(3)$-structure on
$M$ and the abelian system itself, that we ignore at this point of our
discussion.

\subsection{Relation with generalized geometry:
  symmetries}\label{sec:GGstrom}

We address now the relation between generalized geometry and the space
of infinitesimal variations of the Strominger system $H^1(S^*)$,
constructed in Section \ref{sec:infmoduligeneral}. We start our
discussion with the definition of a suitable transitive Courant
algebroid over $M$ and the characterization of its group of
automorphisms. We follow the notation introduced in Section
\ref{sec:parameters}.

We fix a solution $(\Omega,A,\nabla,\omega)$ of the Strominger system
\eqref{eq:stromintro}.  We define the quantities $H := d^c \omega$
and
$$
F = F_A + R_\nabla,
$$
the curvature of $\theta = \nabla \times A$, and consider the
covariant derivative $d^\theta$ induced by $\theta$ in $\ad P$. With
this notation, the \emph{Bianchi identity} (last equation in
\eqref{eq:stromintro}) can be written as
\begin{equation}\label{eq:bianchisimpl}
  d H = c (F \wedge F).
\end{equation}
Consider the vector bundle
\begin{equation}\label{eq:Edef}
  E = T \oplus \ad P \oplus T^*
\end{equation}
endowed with the symmetric pairing
$$
\langle X + r + \xi,Y + t + \eta\rangle = \frac{1}{2}(\eta(X) +
\xi(Y)) + c(r,t),
$$
and the canonical projection
$$
\pi \colon E \to T.
$$
Using the quantities $H$, $\theta$ and $F$, we can endow $\Omega^0(E)$
with a Dorfman bracket
\begin{equation}\label{eq:bracket}
  \begin{split}
    [X+r+\xi,Y+t+\eta]  = {} & [X,Y] + \cL_{X}\eta - i_{Y}d\xi + i_{Y}i_{X}H\\
    & - [r,t] - F(X,Y) + d^\theta_Xt - d^\theta_Y r\\
    & + 2c(d^\theta r,t) + 2c(F(X,\cdot),t) - 2c(F(Y,\cdot),r).
  \end{split}
\end{equation} 
Following \cite{ChStXu}, it can be checked from the Bianchi identity
\eqref{eq:bianchisimpl} that the tuple
$(E,\la\cdot,\cdot\ra,[\cdot,\cdot],\pi)$ satisfies the axioms of a
Courant algebroid (see Definition \ref{def:Courant}). An indirect way
of proving that the previous axioms are satisfied is to construct the
Courant algebroid from reduction \cite{BuCaGu} (see \cite[Section
2]{GF} and \cite{BarHek}). We will use the fact that $E$ is obtained
from reduction of an exact Courant algebroid over $P$ in Corollary
\ref{cor:ringAutE}.

Let $\Aut E$ denote the space of automorphisms of the vector bundle
$E$ that preserve the bracket and the pairing. Given $f \in \Aut E$ we
denote by $\check f \in \Diff$ the diffeomorphism on $M$ that it
covers. The property (D3) gives $\pi\circ f = d \chf \circ \pi$, so
$f$ is compatible with the anchor map.  Note that any $\nu\in
\mathrm{O}(\ad P)$ covering a diffeomorphism $\chf$ defines an
orthogonal transformation of $E\cong T+\ad P+T^*$ given by
$$f_\nu:=\left(
\begin{array}{ccc}
  d \chf &  &  \\
  & \nu &   \\
  & & (d\chf^*)^{-1} 
\end{array}
\right).$$ Orthogonal transformations of $E$ compatible with $\pi$ and
acting as the identity on $\ad P$, and hence covering the identity,
are of the form $$(B,a):=\left(
  \begin{array}{ccc}
    \Id & 0 & 0 \\
    a & \Id & 0 \\
    B-c(a,a)  & -2c(a,\cdot) & \Id 
  \end{array}
\right),$$ with $B\in \Omega^2$, $a\in \Omega^1(\ad P)$. To
characterize the group $\Aut(E)$, we note that an element $\nu \in
O(\ad P)$ covering $\chf \in \Diff$ acts on the covariant derivative
$d^\theta$ by
$$
i_V(\nu \cdot d^\theta)r = \nu \cdot (i_{\chf^*V}d^\theta (\nu^{-1}
\cdot r)),
$$
where we use the action of $\nu$ on a section $r \in \Omega^0(\ad P)$,
given by $\nu \cdot r (x) =\nu(r_{\chf^{-1}(x)})$, for $x \in
M$. Similarly, we have an action on $F$, given by
$$
\(i_W i_V(\nu \cdot F)\)(x) = \nu
(i_{\chf^*W}i_{\chf^*V}F(\chf^{-1}(x))).
$$
Furthermore, any $a \in \Omega^1(\ad P)$ can be used to deform the
connection $\theta$ to a new connection $\theta + a \colon TP \to VP$
on $P$ such that
$$
i_V d^{\theta + a}r = [\theta^\perp V - i_V a,r] = i_V(d^\theta r +
[a,r]),
$$
where the bracket in the middle expression denotes Lie bracket of
vector fields on $P$, while the bracket in the right expression
denotes the bracket on $\ad P$. The induced curvature is (note that
$d^\theta \circ d^\theta = [F, \cdot]$)
$$
F_{\theta + a} = F + d^\theta a + \frac{1}{2}[a,a],
$$
where $[a,a](V,W) = 2[a(V),a(W)]$.

\begin{proposition}\label{prop:Aut(E)}
  The group $\Aut(E)$ is the set of orthogonal transformations
  \begin{align*}
    \Big\{ f_\nu(B,a)  \st  & \nu \in \Aut(\ad P,  \la \cdot,\cdot \ra, [\cdot , \cdot]), B\in \Omega^2, a\in \Omega^1(\ad P),\\ & d^\theta = \nu \cdot d^{\theta+a}, \qquad F  = \nu \cdot F_{\theta + a},\\
    & \chf^*H = H -dB+2c(a,F) + c(a \wedge d^\theta a) +\frac{1}{3}
    c(a \wedge [a,a]) \Big\}
  \end{align*}
  together with the product given, for $f=f_\nu (B,a)$ and
  $f'=f_{\nu'}(B',a')$, by
$$ f \circ f' = f_{\nu \nu'}(\chf'^*B+B'+c((\nu'^{-1} \cdot a)\wedge a'), \nu'^{-1} \cdot a + a').$$
\end{proposition}

\begin{proof}
  Given $f\in \Aut(E)$, the compatibility with the anchor map gives
  that the first row is $(d \chf,0,0)$. Then, as $f$ preserves the
  pairing, we have that the entry $(2,3)$ vanishes and $f$ is of the
  form
$$
f=\left(
  \begin{array}{ccc}
    d \chf & 0  & 0 \\
    * & \nu & 0 \\
    *  & *  & d\chf^{-1} 
  \end{array}
\right),
$$
for some $\nu \in \textrm{O}(\ad P)$. The transformation
$f_{\nu}^{-1}f$ preserves the pairing, is the identity on $\ad P$ and
is compatible with $\pi$, so it has to be $(B,a)$ for some $B\in
\Omega^2$ and $a\in \Omega^1(\ad P)$. Thus, $f=f_\nu(B,a)$. We next
check what are the constraints on $\nu,B,a$ for $f$ to preserve the
Courant bracket.
 
First, $\nu$ must preserve the bracket on $\ad P$. If that condition
is satisfied,
$$f_\nu [f_\nu^{-1}\cdot,f_\nu^{-1}\cdot]_{d^\theta,F,H} = [\cdot,\cdot]_{\nu \cdot d^\theta,\nu \cdot F,\chf_*H }. \label{eq:action-e}$$
On the other hand, after several elementary calculations, we have
$$(B,a)[(-B,-a)\cdot,(-B,-a)\cdot]_{d^\theta,F,H} =[\cdot,\cdot]_{d^{\theta + a},F_{\theta + a},H'}$$ with
\begin{equation}\label{eq:H'}
  H' = H -dB+2c(a,F) + c(a,d^\theta a) +\frac{1}{3} c(a,[a,a]),
\end{equation}
which give the conditions for $\Aut(P)$.

For the composition law, note that for $\nu'\in \OO(\ad P)$,
$$(B,a)\nu'_* = \nu'_* (\chg^*B,(\nu')^{-1}\chg^* a),$$
and that $(B,a)(B',a')=(B+B'+c(a\wedge a'),a+a')$.
\end{proof}

\begin{remark}
  The previous result applies in the more general setup studied in
  \cite{ChStXu}, for an arbitrary transitive Courant algebroid not
  necessarily obtained from reduction. Note that, for pure $a$-field
  transformations, the condition $d^\theta = d^{\theta + a}$ is in
  general very strong, as it forces $a$ to take values in the centre
  of the bundle of Lie algebras.
\end{remark}
For the next result, we use the fact that $E$ is obtained from
reduction of an exact Courant algebroid over $P$.

\begin{corollary}
  The group $\Aut E$ fits into an exact sequence of groups
  \begin{equation}
    \label{eq:AutW}
    \xymatrix{
      0 \ar[r] & \Omega^2_{cl} \ar[r] & \Aut E \ar[r]^q & \Aut TP/G\\
    }
  \end{equation}
  where $\Omega^2_{cl}$ denotes the space of closed $2$-forms on $M$
  and $\Aut TP/G$ is the group of automorphisms of the Lie algebroid
  $TP/G$.
\end{corollary}
\begin{proof}
  When projecting to $TP/G\cong T\oplus\ad P$, the Courant bracket on
  sections of $E$ corresponds to the Lie bracket on sections of
  $TP/G$. From Proposition \ref{prop:Aut(E)}, we see that
  $q^{-1}(\Id)$ consists of the elements such that $\nu =\Id$,
  $\chf=\Id$, $a=0$, that is, of the exponentials of $2$-forms $B$,
  which must be closed by the last defining condition of $\Aut(E)$.
\end{proof}

The group $\Aut P$ can be regarded as a subgroup of $\Aut
(TP/G)$. Given $g\in \Aut P$ covering a diffeomorphism $\chg$, the
corresponding element $dg \in \Aut (TP/G)$ preserves the vertical
part, i.e. $\ad P$, and hence defines an element
$$
\nu^g : = dg_{|\ad P} \in \Aut(\ad P,\langle \cdot,\cdot\rangle,[\cdot
, \cdot]).
$$
The element $dg \in \Aut (TP/G)$ is described, in terms of the
splitting $T\oplus\ad P$, by
\begin{equation}\label{eq:Aut(TP/G)}
  dg =\left(
    \begin{array}{cc}
      d\chg & \\
      & \nu^g   \\
    \end{array}
  \right) \left(
    \begin{array}{cc}
      1 & \\
      g^{-1}\cdot \theta - \theta  & 1   \\
    \end{array}
  \right),
\end{equation}
where the isomorphism $TP/G\cong T\oplus \ad P$ is given by a
connection $\theta$. We denote by $f_g = f_{\nu^g}$ the induced
orthogonal transformation of $E$. The next result is a direct
consequence of Proposition \ref{prop:Aut(E)}.

\begin{corollary}\label{cor:ringAutE}
  Define $\mathring{\Aut E} = q^{-1}(\Aut P)$. Then, there is an exact
  sequence of groups
  \begin{equation}\label{eq:AutEmathring}
    0 \to \Omega^2_{cl} \to \mathring{\Aut E} \to \Aut P.
  \end{equation}
  Elements in $\mathring{\Aut E}$ correspond to transformations
  $f_g(B,a^g)$, with $g \in \Aut P$ covering $\chg$, $B \in \Omega^2$
  and $a^g:= g^{-1}\cdot \theta - \theta$, which satisfy
  \begin{equation}\label{eq:AutEmathring2}
    \chg^* H = H - dB +2c(a^g,F) + c(a^g,d^\theta a^g) +\frac{1}{3} c(a^g,[a^g,a^g]).
  \end{equation}
  Regarding $\mathring{\Aut E} \subset \Aut P \times \Omega^2$, the
  group structure is given by
$$ 
(g,B)(g',B') = (gg',\chg'^*B+B'+c((g'^{-1} \cdot a^g)\wedge a^{g'})),
$$
with Lie algebra $\Lie \; \mathring{\Aut E} \subset \Lie \Aut P \oplus
\Om^2$ given by
$$
\Lie \; \mathring{\Aut E} = \{\hat V + B : \cL_VH = - dB + 2c ((\hat V
\cdot \theta) \wedge F)\},
$$
where $\hat V \cdot \theta$ denotes the infinitesimal action of $\hat
V \in \Lie \Aut P$ on $\theta$ and $V \in \Omega^0(T)$ is the vector
on $M$ covered by $\hat V$.
\end{corollary}

\begin{proof}
  The proof follows from Proposition \ref{prop:Aut(E)}, noticing that
  $g \cdot (\theta + a^g) = \theta$. The claim about the Lie algebra
  follow from taking one-parameter subgroups of generalized
  diffeomorphisms in $\mathring{\Aut E}$.
\end{proof}

\begin{remark}
  Condition \eqref{eq:AutEmathring2} is better understood in the total
  space of $P$, in terms of the transgression formula for the
  Chern-Simons $3$-form for the connection $\theta$ (see
  \cite{GF}). For this, we note that given $g \in \Aut P$, the
  elements in $q^{-1}(g)$ correspond to $B \in \Omega^2$ such that
  \begin{align*}
    d\(B + c(g^{-1}\theta \wedge
    \theta)\) 
    = H - CS(\theta) - g^*(H - CS(\theta)).
  \end{align*}
\end{remark}

\subsection{The complex $\mathring{S}^*$ and the elliptic complex
  $\widehat{S}^*$}\label{sec:ringS*}

We use next the symmetries of the Courant algebroid $E$ to construct
an exact diagram of linear maps
\begin{equation*}
  \xymatrix{
    & 0  \ar[d]   &  &\\
    & H^2(M,\RR)  \ar[d]  &  &\\
    & H^1(\widehat{S}^*)  \ar[d] & &\\
    0 \ar[r] & H^1(\mathring{S}^*) \ar[d] \ar[r] & H^1(S^*) \ar[r]^\delta & H^3(M,\RR)\\
    & 0 &  &
  }
\end{equation*}
such that $H^1(\mathring{S}^*)$ and $H^1(\widehat{S}^*)$ have a
natural interpretation in generalized
geometry. 
To make the link with the construction in Section
\ref{sec:infmoduligeneral}, we need to consider the reduction
$$
P_h \subset P
$$
of the bundle $P$ to $K \times \U(3)$ provided by the hermitian
structure $(\Omega,\omega)$ of the fixed solution. We define $\Lie \;
\mathring{\Aut E}_h \subset \Lie \; \mathring{\Aut E}$ by
$$
\Lie \; \mathring{\Aut E}_h = \{(\hat V,B) \in \Lie \; \mathring{\Aut
  E} : \theta \hat V - \nabla V \in \Omega^0(\ad P_h)\},
$$
where $\theta \hat V$ denotes the vertical part of $\hat V$ with
respect to the connection $\theta$ and $V \in \Omega^0(T)$ is the
vector field covered by $\hat V$.

Let $H^1(S^*)$ be the finite-dimensional vector space of infinitesimal
variations of the given solution (see Section
\ref{sec:infmoduligeneral}). Then, by \eqref{eq:diffnonabelian2} there
is a natural linear map
\begin{equation}\label{eq:fluxmap}
  \delta \colon  H^1(S^*) \to H^3(M,\RR),
\end{equation}
given by
$$
[(\dot \Om,\dot \theta,\dot\om)] \mapsto [J d \dot \om - J
(d\om)^{\dot J J} - 2c (\dot \theta \wedge F)].
$$
We define a sequence
\begin{equation}
  \label{eq:abcircComplex0}
  S^0 \oplus \Omega^2 \lra{\mathring{\mathbf{P}}} S^1 \oplus \Omega^2 \lra{\mathring{\mathbf{L}}} \mathring{\mathcal{R}},
\end{equation}
where we follow the notation of Section
\ref{sec:infmoduligeneral}. Here,
$$
\mathring{\mathcal{R}} := \Omega^{3,1} \oplus \Omega^{2,2} \oplus
\Omega^{0,2}(\ad P_h) \oplus \Omega^3 \oplus \Omega^5 \oplus
\Omega^6(\ad P_h)
$$
and the maps $\mathring{\mathbf{P}}$ and $\mathring{\mathbf{L}}$ are
defined by
\begin{align*}
  \mathring{\mathbf{P}}(V,r,B) = (\mathbf{P}(V,r),-B)
\end{align*}
and
$$
\mathring{\mathbf{L}} = \mathbf{L}_1 \oplus \mathbf{L}_2 \oplus
\mathring{\mathbf{L}}_3 \oplus \mathbf{L}_4 \oplus \mathbf{L}_5,
$$
where
$$
\mathring{\mathbf{L}}_3(\dot \Om,\dot \theta,\dot\om,b) = J d \dot
\om - J (d\om )^{\dot J J} - 2c (\dot \theta\wedge F) - db.
$$
For the next result, we note that we can regard $\Lie \;
\mathring{\Aut E}_h \subset S^0 \oplus \Omega^2$ by using the natural
map
$$
(\hat V,B) \to (V,\theta \hat V - \nabla V,B),
$$
which satisfies
$$
\mathbf{P}(V,\theta \hat V - \nabla V) = (\cL_V\Omega,\hat V \cdot
\theta,\cL_V\omega).
$$

\begin{lemma}
  The sequence \eqref{eq:abcircComplex0} induces a complex
  \begin{equation}
    \label{eq:circComplex}
    (\mathring{S}^*) \qquad \qquad \operatorname{Lie} \mathring{\operatorname{Aut} E}_h \lra{\mathring{\mathbf{P}}} S^1 \oplus \Omega^2 \lra{\mathring{\mathbf{L}}} \mathring{\mathcal{R}}
  \end{equation}
  by restriction of the first arrow, whose cohomology
  $H^1(\mathring{S}^*)$ is naturally isomorphic to the kernel of
  \eqref{eq:fluxmap}. Consequently, the cohomology of $\mathring{S}^*$
  is finite dimensional.
\end{lemma}

The proof follows as in Lemma \ref{lem:mathringA}. Consider now the
vector bundle
$$
E_h = T \oplus \ad P_h \oplus T^*
$$
and note that $\Omega^0(E_h) = S^0 \oplus \Omega^1$. We define a
sequence of differential operators
\begin{equation}
  \label{eq:circComplexinner}
  (\widehat{S}^*) \qquad \qquad \Omega^0(E_h) \lra{\mathring{\mathbf{P}}} S^1 \oplus \Omega^2 \lra{\mathring{\mathbf{L}}} \mathring{\mathcal{R}},
\end{equation}
where
\begin{align*}
  \mathring{\mathbf{P}}(V,r,\xi) = (\mathbf{P}(V,r),d\xi + i_V(d^c\omega) -
  2c(r + \nabla V,F)).
\end{align*}
A straightforward calculation shows that
\begin{equation}\label{eq:innerouter}
  (V,r,-d\xi - i_V(d^c\omega) + 2c(r + \nabla V,F)) \in \operatorname{Lie} \mathring{\operatorname{Aut} E}_h
\end{equation}
and therefore \eqref{eq:circComplexinner} is a complex of differential
operators.

\begin{proposition}
  The sequence \eqref{eq:circComplexinner} is an elliptic complex of
  differential operators of degree $1$. There is an exact sequence
  \begin{equation}
    \label{eq:Scircexactseqtr}
    0 \to H^2(M,\RR) \to H^1( \widehat S^*) \to H^1(\mathring{S}^*) \to 0
  \end{equation}
  where $H^1( \widehat S^*)$ denotes the cohomology of
  \eqref{eq:circComplexinner}.
\end{proposition}

The proof is analogue to the proof of Proposition
\ref{prop:Scircexactseqab}.

\begin{remark}
  The ellipticity of the complex \eqref{eq:circComplexinner} gives an
  alternative proof of the finite-dimensionality of $H^1(\mathring{S}^*)$. 
\end{remark}

\subsection{Relation with generalized geometry:
  metrics}\label{sec:genmetrics}

We finish this section by providing an interpretation of $S^1 \oplus
\Omega^2$---the middle step in the complexes \eqref{eq:circComplex}
and \eqref{eq:circComplexinner}---as a space of infinitesimal
variations of a generalized metric on $E$. For this, we recall a few
basic facts about generalized metrics on the Courant algebroid $E$
following \cite{GF}. This will provide the necessary background for
Section \ref{sec:stgen}.

Let $(t,s)$ be the signature of the pairing on the Courant algebroid
$E$. A generalized metric of signature $(p,q)$, or simply a metric on
$E$, is a reduction of the $O(t,s)$-bundle of frames of $E$ to
$$
O(p,q)\times O(t-p,s-q) \subset O(t,s).
$$
Alternatively, it is given by a subbundle
$$
V_+ \subset E
$$
such that the restriction of the metric on $E$ to $V_+$ is a
non-degenerate metric of signature $(p,q)$. We denote by $V_-$ the
orthogonal complement of $V_+$ on $E$. A generalized metric determines
a vector bundle isomorphism
$$
G\colon E \to E,
$$
with $\pm 1$-eigenspace $V_\pm$, which is symmetric, $G^* = G$, and
squares to the identity, $G^2 = \Id$. The endomorphism $G$ determines
completely the metric, as $V_+$ is recovered by
$$
V_+ = \Ker(G - \Id).
$$

\begin{definition}[\cite{GF}]\label{def:admet}
  A metric $V_+$ of arbitrary signature is admissible if
$$
V_+ \cap T^* = \{0\} \qquad \textrm{and} \qquad \rk V_+ = \rk E - \dim
M.
$$
\end{definition}

We recall the basic result that we need in order to understand the
parameters encoded by an admissible generalized metric.

\begin{proposition}[\cite{GF}]\label{prop:admissmetric2}
  An admissible metric $V_+$ on $E$ is equivalent to a pair given by a
  metric $g$ on $M$ and an isotropic splitting. Using the canonical
  isotropic splitting of \eqref{eq:Edef}, $V_+$ is given by a metric
  $g$ on $M$ and an orthogonal transformation by $(b,a) \in \Omega^2
  \oplus \Omega^1(\ad P)$
  \begin{equation}
    V_+ = (-b,-a)\{X + g(X) + r \colon X \in T, r \in \ad P\}.
  \end{equation}
  An admissible metric determines a connection $\theta' = \theta + a$
  on $P$ and a $3$-form $H'$ on $M$ given by \eqref{eq:H'}, such that
  the bracket in the splitting provided by $V_+$ is given by
  \eqref{eq:bracket}.
\end{proposition}

From the previous result, we note that the fixed solution of the
Strominger system determines an admissible generalized metric $V_+$,
given by the metric determined by the $\SU(3)$-structure
$(\Omega,\omega)$ and the canonical isotropic splitting, with
connection $\theta$ and $3$-form $H = d^c\omega$. Furthermore,
cocycles for $\mathring{S}^*$ in $S^1 \oplus \Omega^2$ correspond to
infinitesimal variations of the admissible metric $V_+$.

Therefore, the finite-dimensional vector space $H^1(\mathring{S}^*)$
corresponds to a space of infinitesimal variations of $V_+$ as a
generalized metric modulo natural symmetries of the smooth Courant
algebroid $E$ \eqref{eq:Edef}, while $H^1(\widehat{S}^*)$ is cut out
by \emph{inner symmetries} of $E$.

\section{Killing spinors in generalized geometry}\label{sec:stgen}

In this section we introduce a natural notion of Killing spinor in
generalized geometry. The \emph{Killing spinor equations}
\begin{equation}\label{eq:Killing}
  D^\phi_+ \eta = 0, \qquad \slashed D_-^\phi \eta = 0,
\end{equation}
depend on a generalized metric $V_+$ on a smooth Courant algebroid $E$
over $M$ and a smooth function $\phi \in C^{\infty}(M)$. We specify to
the case of a spin manifold of dimension six and study the equations
\eqref{eq:Killing} in two different cases. Firstly, when $E$ is exact,
we show that a solution of \eqref{eq:Killing} is equivalent to a
metric on $M$ with holonomy $\SU(3)$. Secondly, for a suitable choice
of transitive Courant algebroid $E$, by applying the original result
of Strominger and Hull \cite{HullTurin,Strom}, we prove that the Killing spinor equations
\eqref{eq:Killing} correspond to the Strominger system.

The present section gives a precise interpretation of the vector
spaces $H^1(\mathring{S}^*)$ and $H^1(\widehat{S}^*)$ constructed in
Section \ref{sec:anomalyflux}, as spaces of infinitesimal deformations
of solutions of the Killing spinor equations \eqref{eq:Killing} modulo
infinitesimal symmetries of a Courant algebroid, offering a conceptual
explanation of the appearance of generalized geometry in the study of
the Strominger system. In addition, it provides a unifying framework
for the theory of the Strominger system and the well-established
theory of metrics with holonomy $\SU(3)$, which we expect will have
future applications in the former.

\subsection{The canonical Levi-Civita connection}

Let $E$ be a smooth, transitive, Courant algebroid over a smooth
manifold $M$ of arbitrary dimension, that is, a vector bundle
satisfying the axioms of Definition \ref{def:Courant} and such that
the anchor map
$$
\pi \colon E \to T
$$
is surjective.  In this section we introduce a canonical notion of
\emph{Levi-Civita connection} on $E$---a natural torsion-free
connection associated to an admissible generalized metric. For
simplicity, we will assume that $E$ is obtained from reduction of an
exact Courant algebroid on a principal $G$-bundle $P$, as in
\cite[Sec. 2]{GF}. The general case follows easily from \cite{ChStXu}.

Recall that a generalized connection $D$ (or simply, a connection) on
$E$ is a first order differential operator
$$
D \colon \Omega^0(E) \to \Omega^0(E^* \otimes E)
$$
satisfying the Leibniz rule $D_e(\phi e') = \phi D_ee' +
\pi(e)(\phi)e'$, for $e,e' \in \Omega^0(E)$ and $\phi\in
C^\infty(M)$. We will only consider connections compatible with the
inner product on $E$, that is, satisfying
$$
\pi(e)(\langle e',e'' \rangle) = \langle D_e e',e'' \rangle + \langle
e',D_e e'' \rangle.
$$
The space of connections on $E$ is an affine space modelled on
$\Omega^0(E^*\otimes \fo(E))$.

Following \cite{G3}, we first introduce the notion of
\emph{Gualtieri--Bismut connection}. Given an admissible metric $V_+$
(see Definition \ref{def:admet}), we define
$$C^+=(\ad P)^\perp \subset V_+.$$
We can associate an endomorphism of the vector bundle $E$ such that
$C(V_+) = V_-$ and $C(V_-) = C_+$, defined by
$$
C = \pi_{|V_-}^{-1} \circ \pi \circ \Pi_+ + \pi_{|C_+}^{-1} \circ \pi
\circ \Pi_-,
$$
where
$$
\Pi_{\pm} = \frac{1}{2}(\Id \pm G) \colon E \to V_{\pm}
$$
denote the orthogonal projections. In the canonical splitting provided
by $V_+$
\begin{equation}\label{eq:Esplittr}
  E \cong T \oplus \ad P \oplus T^*
\end{equation}
we have (see Proposition \ref{prop:admissmetric2})
\begin{equation}\label{eq:V+b0}
  V_+ = \{X + g(X) + r \colon X \in T, r \in \ad P\},
\end{equation}
and then we can write explicitly
$$
C(X + r + gX) = X- gX \quad \textrm{and} \quad C(Y - gY) = Y + gY.
$$

\begin{definition}
  The \emph{Gualtieri--Bismut connection} $D^B$ of $V_+$ on $E$ is
  defined by
  \begin{equation}\label{eq:bismut}
    D^B_{e}e' = [e_-,e'_+]_+ + [e_+,e'_-]_- + [C e_-,e'_-]_- + [Ce_+,e'_+]_+ ,
  \end{equation}
  where $e_\pm = \Pi_\pm e$.
\end{definition}
We check now that $D^B$ is well-defined.
\begin{proposition}
  The expression \eqref{eq:bismut} defines a well-defined connection
  on $E$ compatible with $V_+$, that is, $D^B$ sends $V_\pm$ to
  $V_\pm$.
\end{proposition}
\begin{proof}
  The Leibniz rule and the facts that $D^B$ preserves $V_\pm$ and the
  inner product on $E$ follow exactly as in \cite[Th. 3.1]{G3}. To
  prove that $D^B_{\phi e}e' = \phi D_e^Be'$ for any $\phi \in
  C^\infty(M)$, using the properties of the Dorfman bracket we
  calculate
  \begin{align*}
    D^B_{fe_-}e' & {} =  [f e_-,e'_+]_+ + [C(f e_-),e'_-]_-\\
    & {} = - f [e'_+,e_-]_+ - \pi(e'_+)(f)(e_-)_+\\
    & \phantom{ {} =}- f [e'_-,Ce_-]_- - \pi(e'_-)(f)(C e_-)_-\\
    & {} = f([e_-,e'_+]_+ + [Ce_-,e'_-]_-)
  \end{align*}
  and similarly for $D^B_{fe_+}e'$.
\end{proof}

The generalized torsion \cite{G3} of a connection $D$ on $E$ is a
totally skew tensor $T_D \in \Lambda^3 E^*$ defined by
$$
T_D(a,b,c) = \langle D_{a}b - D_{b}a - [[a,b]],c \rangle +
\frac{1}{2}\(\langle D_{c} a,b \rangle - \langle D_{c} b,a \rangle\),
$$
where
$$
[[a,b]] = \frac{1}{2}([a,b] - [b,a])
$$
is the skew-symmetrization of the Dorfman bracket on sections of
$E$. A connection with vanishing torsion will be referred as a
\emph{torsion-free connection}. By analogy with hermitian geometry, we
introduce the following notion of Levi-Civita connection associated to
a generalized metric.

\begin{definition}
  The \emph{canonical Levi-Civita connection} of $V_+$ is defined by
  \begin{equation}\label{eq:levi-civita}
    D^{LC} = D^B - \frac{1}{3} T_{D^B},
  \end{equation}
  where we identify the torsion $T_{D^B}$ with the element $\langle
  \cdot, \cdot \rangle^{-1} T_{D^B} \in E \otimes \Lambda^2 E^*$.
\end{definition}
By construction, $D^{LC}$ is a natural object on $E$, that is, given
an automorphism $f \in \Aut E$ and a generalized metric $V_+$, we have
$$
f_*(D^{LC}(V_+)) = D^{LC}(f(V_+)).
$$
As in \cite{GF}, we can modify $D^{LC}$ by elements in $E^*$, while
preserving the torsion-free property---hence, torsion-free, metric
connections are not unique (see also \cite{CSW}).  Note that any other
compatible generalized connection differs from $D^{LC}$ by
$$
\chi \in E^* \otimes \(\mathfrak{o}(V_+)\oplus \mathfrak{o}(V_-)\).
$$
Given $\varphi \in E^*$, consider the \emph{Weyl term}
$$
\chi^\varphi \in E^* \otimes\mathfrak{o}(E)
$$
defined by
\begin{equation}\label{eq:Weylterm}
  \chi^\varphi_e e' = \varphi(e') e - \langle e,e' \rangle \langle\cdot,\cdot\rangle^{-1}\varphi.
\end{equation}
We introduce the notation
$$
\chi^{\pm\pm\pm}_{e}e' = \Pi_{\pm} \chi_{e_{\pm}}(e'_{\pm}).
$$
for $\chi \in E^* \otimes (\mathfrak{o}(E))$ and $e, e' \in E$. Then,
we obtain a new torsion-free, compatible, generalized connection by
the formula
\begin{equation}\label{eq:connectionvarphi}
  D^\varphi = D^{LC} + \frac{1}{3(\rk V_+ -1)}(\chi^\varphi)^{+++} + \frac{1}{3(\rk V_- -1)}(\chi^\varphi)^{---}.
\end{equation}
Once more, we have that the connection $D^\varphi$ is natural, in the following sense
$$
f_*(D^{f^* \varphi}(V_+)) = D^{\varphi}(f(V_+)),
$$
for any $f \in \Aut E$.

\begin{remark}
  A more natural approach to the previous construction would be to
  derive $D^\varphi$ as the canonical Levi-Civita connection on a
  modified version of a Courant algebroid, which keeps track of
  conformal changes in the pairing $\langle \cdot,\cdot \rangle$.
\end{remark}

\subsection{Some explicit formulae}

We fix a generalized metric $V_+$ on $E$ and consider the associated
splitting \eqref{eq:Esplittr} and triple $(g,H,\theta)$ (see
Proposition \ref{prop:admissmetric2}). We denote by $c$ the symmetric
pairing induced by $\langle \cdot,\cdot \rangle$ in $\ad P$. We define
connections on $T$ with skew torsion, compatible with the metric $g$,
given by
\begin{equation}\label{eq:nabla+3}
  \begin{split}
    \nabla^\pm &= \nabla^g \pm \frac{1}{2}g^{-1}H,\\
    \nabla^{\pm1/3} &= \nabla^g \pm \frac{1}{6}g^{-1}H,
  \end{split}
\end{equation}
where $\nabla^g$ denotes the Levi-Civita connection of the metric $g$
on $M$. Consider the covariant derivative $d^\theta$ on $\ad P$
determined by $\theta$. Considering
\begin{equation}\label{eq:abcd}
  \begin{split}
    a_+ &= X + r + gX,\\
    b_- &= Y - gY,\\
    c_+ &= Z + t + gZ,\\
    d_- &= W - gW,
  \end{split}
\end{equation}
we have
\begin{equation}\label{eq:bismutexp}
  \begin{split}
    D^B_{a_+} c_+ &= 2\Pi_+\(\nabla^+_XZ + g^{-1}c(i_XF,t)\) + d^\theta_X t - F(X,Z),\\
    D^B_{b_-} c_+ &= 2\Pi_+\(\nabla^+_YZ + g^{-1}c(i_YF,t)\) + d^\theta_Y t - F(Y,Z),\\
    D^B_{a_+} b_- &= 2\Pi_-\(\nabla^-_XY + g^{-1}c(i_YF,r)\),\\
    D^B_{b_-} d_- &= 2\Pi_-\(\nabla^-_YW\),
  \end{split}
\end{equation}
where $F$ is the curvature of $\theta$. To calculate the torsion of
$D^B$, consider the auxiliary covariant derivative
$$
D' = \nabla^g \oplus d^\theta \oplus \nabla^{g^*}
$$
on $E$, compatible with $V_+$. For $e = X + r + \xi$, define
\begin{equation*}
  \begin{split}
    \chi'_e &= - \langle \cdot,\cdot\rangle^{-1}T_{D'} = \left(
      \begin{array}{ccc}
        0 & 0 & 0\\
        - i_X F & - c^{-1}(c(r,[\cdot,\cdot])) & 0\\
        i_X H - 2c(F,r) & 2c(i_X F,\cdot) & 0
      \end{array}\right) \in \mathfrak{o}(E),
  \end{split}
\end{equation*}
(note that $\chi' = \chi^0$ in the notation of \cite{GF}). We also
define
$$
\chi'_C \in E^* \otimes \mathfrak{o}(E)
$$ 
by $\chi'_{Ce} = \chi'_{C(e)}$, which is explicitly given by
$$
\chi'_{Ce} = \left(
  \begin{array}{ccc}
    0 & 0 & 0\\
    - i_X F & 0 & 0\\
    i_X H  & 2c(i_X F,\cdot) & 0
  \end{array}\right) \in \mathfrak{o}(E).
$$
Then, we have
\begin{equation}\label{eq:bismutchi}
  D^B = D' + (\chi'_C)^{+++} + (\chi')^{---} + (\chi')^{+-+} + (\chi')^{-+-}.
\end{equation}
With the previous formula, a direct calculation using \cite[Lemma
3.5]{GF} leads us to the following expression for the torsion. Let
$$
\pi_Q \colon E \to Q = E/T^* \cong TP/G
$$ 
be the natural projection and denote by $CS(\theta) \in \Omega^3(P)$
the Chern-Simons $3$-form of $\theta$.
\begin{lemma}
  The torsion $T_{D^B}$ is the element of $\Lambda^3V_+^* \oplus
  \Lambda^3V_-^*$ given by
  \begin{equation}\label{eq:bismuttorsion}
    T_{D^B} = \pi_{Q|V_+}^*\(H - CS(\theta)\) + \pi_{|V_-}^*H. 
  \end{equation}
\end{lemma}
More explicitly, taking $a_+,c_+$ as in \eqref{eq:abcd} and $b_+ = Y +
s + gY$, we obtain the formula
\begin{align*}
  T^+_{D^B}(a_+,b_+,c_+) & {} = H(X,Y,Z) + c(r,[s,t])\\
  & \phantom{ {} =} - c(F(X,Y),t) + c(F(X,Z),s) - c(F(Y,Z),r),
\end{align*}
for $T_{D^B} = T_{D^B}^+ + T_{D^B}^-$ the natural
decomposition. Similarly, setting $a_- = X - gX$ we have
$$
T^-_{D^B}(a_-,b_-,d_-) = H(X,Y,W).
$$
We calculate now an explicit formula for the Levi-Civita connection of
$V_+$. We have
$$
\langle \cdot,\cdot\rangle^{-1} T_{D^B} = (\chi^B)^{+++} +
(\chi^B)^{---},
$$
where
\begin{equation*}
  \begin{split}
    \chi^B_e & = \left(
      \begin{array}{ccc}
        0 & 0 & 0\\
        - i_X F & c^{-1}(c(r,[\cdot,\cdot])) & 0\\
        2 i_X H - 2c(F,r) & 2c(i_X F,\cdot) & 0
      \end{array}\right) \in \mathfrak{o}(E).
  \end{split}
\end{equation*}
Therefore
\begin{equation}\label{eq:LeviCivitachi}
  D^{LC} {} = D^B - \frac{1}{3} (\chi^B)^{+++} - \frac{1}{3} (\chi^B)^{---}
\end{equation}
and we conclude that
\begin{equation}\label{eq:Levi-Civitaexp}
  \begin{split}
    D^{LC}_{a_+} c_+ & {} = 2\Pi_+\(\nabla^{1/3}_XZ + \frac{2}{3}g^{-1}c(i_XF,t) + \frac{1}{3}g^{-1}c(i_ZF,r)\)\\
    & \phantom{ {} = }+ d^\theta_X t - \frac{2}{3}F(X,Z) - \frac{1}{3}c^{-1}c(r,[t,\cdot]),\\
    D^{LC}_{b_-} c_+ &= 2\Pi_+\(\nabla^+_YZ + g^{-1}c(i_YF,t)\) + d^\theta_Y t - F(Y,Z),\\
    D^{LC}_{a_+} b_- &= 2\Pi_-\(\nabla^-_XY + g^{-1}c(i_YF,r)\),\\
    D^{LC}_{b_-} d_- &= 2\Pi_-\(\nabla^{-1/3}_YW\).
  \end{split}
\end{equation}
\begin{remark}\label{rem:D0}
  The connection $D^{LC}$ differs by the torsion-free connection $D^0$
  constructed in \cite{GF} by the term $\frac{1}{3} (\chi'')^{+++}$,
  where
  \begin{equation}\label{eq:errorterm}
    \begin{split}
      \chi''_e & = \left(
        \begin{array}{ccc}
          0 & 0 & 0\\
          - i_X F & 0 & 0\\
          4c(F,r) & 2c(i_X F,\cdot) & 0
        \end{array}\right) \in \mathfrak{o}(E)
    \end{split}
  \end{equation}
  (that vanishes identically in the exact case), as it follows from
  \begin{equation}
    \begin{split}
      D^{LC} = {} &  D' + \frac{1}{3} (\chi')^{+++} + \frac{1}{3} (\chi'')^{+++}\\
      & + \frac{1}{3} (\chi')^{---} + (\chi')^{+-+} + (\chi')^{-+-}.
    \end{split}
  \end{equation}
  This provides a new example of a pair of different torsion-free
  connections on $E$ compatible with the same metric
  (cf. \cite{CSW}). We note that the connection $D^0$ is not natural
  for the action of $\Aut E$ (this disproves a claim in \cite[Remark
  3.8]{GF}). For example, for $(B,a) \in \Aut E$ we have
  \begin{align*}
    \chi'' ((B,a) \cdot V_+)  = {} & \chi''(V_+)\\
    (B,a)^{-1} \(\chi''_{(B,a) e}(V_+)\)(B,a)e'  = {} & \(\chi''_{e}(V_+)\)e' - 2c(a,F(X,Y))\\
    & + 4c(i_YF,a(X)) + 2c(i_XF,a(Y))
  \end{align*}
  for $e = X + r + \xi$ and $e' = Y + t + \eta$.
\end{remark}

Finally, we provide an explicit formula for the connection
\eqref{eq:connectionvarphi}. In this work, we are interested in the
case that $\varphi$ is an exact $1$-form, so we assume $\varphi \in
\Omega^1(M)$. Then, $D^\varphi$ satisfies
\begin{equation}\label{eq:Dvarphi}
  \begin{split}
    D^\varphi_{a_+} c_+ &= D^{LC}_{a_+}c_+ + \frac{(\rk V_+ -1)^{-1}}{3}\Pi_+\(\varphi(Z)a_+ - 2(g(X,Z) + c(r,t))\varphi\),\\
    D^\varphi_{b_-} c_+ &= D^{LC}_{b_-} c_+,\\
    D^\varphi_{a_+} b_- &= D^{LC}_{a_+} b_-,\\
    D^\varphi_{b_-} d_- &= D^{LC}_{b_-} d_- + \frac{(\rk V_-
      -1)^{-1}}{3}\Pi_-\(\varphi(W)b_- + 2g(Y,W)\varphi\).
  \end{split}
\end{equation}

\subsection{The Killing spinor equations}

We introduce now the notion of Killing spinor of our interest and
study its relation with classical geometry.  We will restrict
ourselves to the case when $M$ is a six dimensional spin manifold and
$E$ is transitive or exact.

As in the previous section, let $E$ be a transitive Courant algebroid,
obtained from reduction.  Consider an admissible generalized metric
$V_+\subset E$ such that the corresponding metric on $M$ (see
\eqref{eq:V+b0}) is positive definite. Then, there exists a
positive-definite metric $g$ on $M$ such that
\begin{equation}
  \begin{split}
    V_+ & = \{X + g(X) + r \colon X \in T, r \in \ad P\},\\
    V_- & = \{X - g(X) \colon X \in T\},
  \end{split}
\end{equation}
and $\pi_{|V_-} \colon V_- \to (T,g)$ induces an anti-isometry. The
spin condition $w_2(T) = 0$ for the manifold $M$ implies the existence
of a spinor bundle $S(V_-)$. As $\rk V_-$ is even, we have a direct
sum decomposition into positive and negative chirality half-spinor
bundles
$$
S(V_-) = S_+(V_-) \oplus S_-(V_-) \subset Cl(V_-).
$$
Here, $Cl(V_-)$ is the Clifford bundle for $V_-$, with fibre
$Cl((V_{-|x})^*)$ for any $x \in X$.

Given $\phi \in C^\infty(M)$, consider the generalized connection
$D^\phi = D^\varphi(V_+)$ determined by $V_+$ and the $1$-form
$$
\varphi = 6d\phi.
$$
By compatibility, $D^\phi$ induces differential operators
$$
D^\phi_{\pm}: V_-\to V_-\ot (V_\pm)^*,
$$
where we omit, here and below, the symbol of sections for the sake of
simplicity. From $D^\phi_+$ and $D^\phi_-$ we get differential
operators on spinors
$$
D^\phi_\pm:S_+(V_-)\to S_+(V_-)\ot (V_\pm)^*
$$
and the associated Dirac operator
$$
\slashed D^\phi_-:S_+(V_-)\to S_-(V_-).
$$

\begin{definition}
  Given a generalized metric $V_+$, as before, and $\phi \in
  C^{\infty}(M)$, the \emph{Killing spinor equations} for a spinor
  $\eta \in S_+(V_-)$ are given by
  \begin{equation}\label{eq:Killing2}
    \begin{split}
      D^\phi_+ \eta &= 0,\\
      \slashed D_-^\phi \eta & = 0.
    \end{split}
  \end{equation}
\end{definition}

\begin{proposition}\label{prop:naturality}
  The system \eqref{eq:Killing2} is a natural system of equations in
  generalized geometry, that is, solutions are exchanged under
  generalized diffeomorphisms.
\end{proposition}

\begin{proof}
  Given a triple $(V_+,\phi,\eta)$ which satisfies \eqref{eq:Killing2}
  and $f \in \Aut E$, we have to check that the triple $(f(V_+),\check
  f_*\phi,f_* \eta)$ is also a solution of the equations. Here,
  $\check f$ is the diffeomorphism on $M$ covered by $f$ (see
  \eqref{eq:AutEmathring}). This follows from the naturality of the
  Bismut connection $D^B$, the torsion $T_{D^B}$ and the Weyl term
  $\chi^\varphi$.
\end{proof}

In the next result we provide a more explicit characterization of the
Killing spinor equations, that will be applied in Section
\ref{sec:charStrom} to the Strominger system. Recall that given a
generalized metric $V_+$, by Proposition \ref{prop:admissmetric2} we
obtain a corresponding triple $(g,H,\theta)$.

\begin{lemma}\label{lem:charStrom}
  The Killing spinor equations \eqref{eq:Killing2} are equivalent to
  the system
  \begin{equation}\label{eq:Killingclassicaltr}
    \begin{split}
      F \cdot \eta &= 0\\
      \nabla^- \eta &= 0,\\
      (H + 2 d\phi)\cdot \eta & = 0,\\
      dH - c(F \wedge F) & = 0,
    \end{split}
  \end{equation}
\end{lemma}
\begin{proof}
  We use the anchor map $\pi_{|V_-} \colon V_- \to (T,g)$ to identify
  $S(V_-)$ with $S(T)$, so that $F,H$ and $d\phi$, as sections of
  $\wedge^\bullet T^* \subset Cl(T)=\End(S(T))$, act on the
  spinor~$\eta$.

  Let $\{e_1,\ldots,e_n\}$ be a local orthonormal frame for $T$. We use the following convention for the Clifford algebra relations
$$
e_ie_j + e_je_i = 2\delta_{ij}.  
$$ 
An endomorphism $A\in \End(T)$ satisfies
$$ A = \sum_{i,j=1}^n g(Ae_i,e_j)e^i\ot e_j,$$
for $\{e^j\}$ the dual frame of $\{e_j\}$. Since $e^i\ot e_j - e^j\ot
e_i\in \fso(T)$ embeds as $\frac{1}{2}e^j e^i$ in the Clifford
algebra, an endomorphism $A\in \fso(T)$ corresponds to
\begin{equation}\label{eq:A-as-element-in-Cl}
  A=\frac{1}{2}\sum_{i < j} g(Ae_i,e_j) e^j e^i \in \Cl(T).
\end{equation}
Using \eqref{eq:LeviCivitachi} combined with \eqref{eq:Levi-Civitaexp}
and \eqref{eq:Dvarphi}, we have
\begin{align*}
  D^\phi_+ \eta & = \nabla^g \eta  + \frac{1}{2}\sum_{i < j} g((\chi')^{-+-}e_i,e_j) e^j e^i \cdot \eta,\\
  & = \nabla^g \eta - \frac{1}{2}\sum_{i < j} H(e_i,e_j,\cdot) e^j e^i \cdot \eta + \sum_{i < j} c(F(e_i,e_j),\cdot) e^j e^i \cdot \eta\\
  & = \nabla^- \eta - c(F \cdot \eta,\cdot),
\end{align*}
and therefore the vanishing of $D^\phi_+ \eta$ is equivalent to the
first two equations in \eqref{eq:Killingclassicaltr}.  Similarly,
\begin{align*}
  D^\phi_- \eta &{} =  \nabla^- \eta + \frac{1}{3}\sum_{i < j} H(e_i,e_j,\cdot) e^j e^i \cdot \eta \\
  & \phantom{ {}= }+ \frac{1}{5}\sum_{i < j} \(d\phi(e_i) e^j e^i \eta
  \ot e^j - d\phi(e_j)e^j e^i \eta \ot e^i\)
\end{align*}
and hence, assuming that $D^\phi_+ \eta = 0$, we have
\begin{align*}
  \slashed D^\phi_- \eta & = \frac{1}{3}\sum_{i < j} H(e_i,e_j,e_k) e^k e^j e^i \cdot \eta - 2 d\phi \cdot \eta \\
  & = - (H + 2 d\phi)\cdot \eta.
\end{align*}
The last equation follows from the compatibility condition between $H$
and $\theta$ for any admissible metric (see Proposition
\ref{prop:admissmetric2}).
\end{proof}

\subsection{Hull-Strominger's theorem}

We study next the consequences of the existence of a Killing spinor \eqref{eq:Killing2} in terms of complex hermitian
geometry. The link 
is provided by Lemma \ref{lem:charStrom} combined with the following
result, originally due to Hull and Strominger \cite{HullTurin,Strom}, which constitutes
the starting point for the theory of the Strominger system. We give
here a different proof, using Gauduchon's formula for the Bismut
connection on the canonical bundle \cite{Gau}.

\begin{theorem}[\cite{Strom}]\label{thm:Strom}
  Given a real $3$-form $H \in \Omega^3$ and a real smooth function
  $\phi \in C^\infty(M)$, a solution $(g,\eta)$ of the system
  \begin{equation}\label{eq:thmStrom1}
    \begin{split}
      \nabla^- \eta &= 0,\\
      (H + 2 d\phi) \cdot \eta & = 0,
    \end{split}
  \end{equation}
  for a non-vanishing half-spinor $\eta \in S_+$ is equivalent to an
  $\SU(3)$-structure $(\omega,\Omega)$ on $M$ with integrable almost
  complex structure $J$, K\"ahler form $\omega = g(J\cdot,\cdot)$ and
  holomorphic volume form $\Omega$, satisfying
  \begin{equation}\label{eq:thmStrom2}
    \begin{split}
      H &= d^c\omega,\\
      \phi & = -\frac{1}{2}\log \|\Omega\|_\omega - \kappa,\\
      d^*\omega & = d^c \log \|\Omega\|_\omega,
    \end{split}
  \end{equation}
  for a suitable real constant $\kappa$.
\end{theorem}
\begin{proof}
  We start with a solution $(g,\eta)$ of \eqref{eq:thmStrom1}, and
  note that the first equation reduces the holonomy of $\nabla^-$ to
  $\SU(3)$. This follows from the equality $Spin(6) = \SU(4)$ (see also
  \cite[Lem. 9.15]{LaMi} and \cite[Rem. 9.12]{LaMi}). Let
  $(\psi,\omega)$ be the corresponding $\SU(3)$-structure on $M$.
  Here $\psi$ is a parallel complex $3$-form which determines an
  orthogonal almost complex structure $J$ via \eqref{eq:T01} and
  $\omega = g(J\cdot,\cdot)$ is the corresponding (parallel) K\"ahler
  form.

Using the $\SU(3)$-structure we have an explicit model for the half-spinor bundle $S_+$: the Clifford module
$$
S_+\cong \Lambda^{0,even}
$$
with Clifford action
$$
\xi \cdot \sigma = \sqrt{2} ( i_{(g^{-1} \xi^{1,0})} \sigma +
\xi^{0,1} \wedge \sigma)
$$
for $\xi$ a 
$1$-form and $\sigma \in \Lambda^{0,even}$. For
$\SU(3)$, the space of even parallel spinors is $1$-dimensional and
$\eta$ is identified with a non-vanishing function (\cite{Wang}),
which we can assume to be $1$.

Choosing a basis $\{dz_j,d\overline{z}_j\}$ of $(1,0)$ and $(0,1)$
forms at a point such that
$$
g = \sum_{j=1}^3 dz_j \otimes d\overline{z}_j + d\overline{z}_j
\otimes dz_j
$$
we have
\begin{align*}
  dz_j \cdot 1 &= 0,\\
  d\overline{z}_j \cdot 1 & = \sqrt{2} d\overline{z}_j,\\
  dz_i \wedge dz_j \wedge d\overline{z}_k \cdot 1 & = 0,\\
  d\overline{z}_i \wedge d\overline{z}_j \wedge dz_k \cdot 1 & =
  \sqrt{2}(\delta_{ki} d\overline{z}_j - \delta_{kj} d\overline{z}_i).
\end{align*}
This implies
$$
(H + 2 d\phi) \cdot 1 = 2 \sqrt{2} \(H^{0,3} + \frac{1}{2}\sum_{i<j}
(H^{1,2}_{\overline{i}\overline{j}i}d\overline{z}_j -
H^{1,2}_{\overline{i}\overline{j}j}d\overline{z}_i) + \dbar\phi\)
$$
where $H = H^{3,0} + H^{2,1} + H^{1,2} + H^{0,3}$ is the decomposition
in types with respect to $J$ and
$$
H^{1,2} = \sum_{i<j}H^{1,2}_{\overline{i}\overline{j}k}
d\overline{z}_i \wedge d\overline{z}_j \wedge dz_k.
$$
Then, it follows that the second equation in \eqref{eq:thmStrom1} is
equivalent to the two conditions
\begin{equation}\label{eq:thmStromproof1}
  \begin{split}
    H^{0,3} & = 0,\\
    i\Lambda_\omega H^{1,2} & = -2 \dbar \phi.
  \end{split}
\end{equation}
Using now that $\nabla^-$ is unitary and has totally skew-torsion
$-H$, by \cite[Eq. (2.5.2)]{Gau},
$$
H = -N + (d^c\om)^{2,1+1,2},
$$
where $N$ denotes the Nijenhuis tensor of $J$, which is of type $(3,0)
+ (0,3)$ \cite[Prop. 1]{Gau}.

Hence, $N = 0$, $H = d^c\omega$ and also
$$
\Lambda_\omega d\omega = 2d\phi,
$$
where, recall, $\Lambda_\omega d\omega = Jd^*\omega$ is the Lee form
of the hermitian structure, and therefore $\omega$ is conformally
balanced. Using now Gauduchon's formula \cite[Eq. (2.7.6)]{Gau}
\begin{equation}\label{eq:nablaC}
  \nabla^C = \nabla^- + i d^*\omega \otimes \Id,
\end{equation}
relating the connections induced by $\nabla^-$ and the Chern
connection $\nabla^C$ on the canonical bundle, combined with $\nabla^-
\psi = 0$, it follows that
$$
\Omega := e^{-2\phi}\psi
$$ 
is a holomorphic volume form for the given complex structure. Finally,
from \eqref{eq:norm3form} we obtain that
\begin{equation}\label{eq:phiSU3}
  \phi = -\frac{1}{2}\(\log \|\Omega\|_\omega - \log \|\psi\|_\omega\)
\end{equation}
where $\|\psi\|_\omega$ is constant and therefore
$$
d^*\omega = d^c \log \|\Omega\|_\omega.
$$
For the converse, we simply define $\eta = 1$ in the model for $S_+$
provided by the $\SU(3)$-structure.
\end{proof}

\subsection{Metrics with holonomy $\SU(3)$ and the Strominger
  system}\label{sec:charStrom}

We want to give a complete characterization of \eqref{eq:Killing2} in
terms of complex geometry using Hull-Strominger's theorem, when $E$ is
either a exact or a transitive Courant algebroid.

We assume first that $E$ is exact.  By definition, $E$ fits into an
exact sequence
$$
0 \to T^* \to E \to T \to 0
$$
and the admissible metric $V_+$ induces a splitting
$$
E = T \oplus T^*.
$$
In this splitting, the generalized metric takes the form
\begin{align*}
  V_+ & = \{X + g(X) \colon X \in T\},\\
  V_- & = \{X - g(X) \colon X \in T\},
\end{align*}
and the induced $3$-form $H$ is closed
$$
dH =0.
$$

\begin{theorem}\label{th:SU3}
  Assume that $E$ is exact. Then $(V_+,\phi,\eta)$ is a solution of
  \eqref{eq:Killing2} with $\eta \neq 0$ if and only if $H =0$, $\phi$
  is constant and $g$ is a metric with holonomy contained in $\SU(3)$.
\end{theorem}

The proof will follow from Hull-Strominger's theorem and Lemma
\ref{lem:charStrom}, which in the present setup specifies to the
following result.

\begin{lemma}\label{lem:SU3}
  Assume that $E$ is exact. Then \eqref{eq:Killing2} is equivalent to
  the system
  \begin{equation}\label{eq:Killingclassical}
    \begin{split}
      \nabla^- \eta &= 0,\\
      (H + 2 d\phi)\cdot \eta & = 0,\\
      dH & = 0.
    \end{split}
  \end{equation}
\end{lemma}

Theorem \ref{th:SU3} is a well-known fact in the literature, that
originally appeared in
\cite{IvanovPapadopoulos,IvanovPapadopoulosnogo}, but we sketch here a
proof with the complete argument for the convenience of the reader.

\begin{proof}[Proof of Theorem \ref{th:SU3}]
  For the `only if' part, we note that from Theorem \ref{thm:Strom}
$$
\Lambda_\omega \rho^C = -2\Lambda_\omega dd^c \phi,
$$
where $\rho^C$ is the Ricci form of the Chern connection
$\nabla^C$. On the other hand, \eqref{eq:nablaC} implies that
\cite[(2.11)]{AlexIvanov}
$$
\Lambda_\omega \rho^C = \Lambda_\omega \rho^B + 4 d^*(Jd^*\omega) + 8
|d^*\omega|^2_\omega,
$$
where $\rho^B$ is the Ricci form of the Bismut connection, and from
\cite[Th. 1.1]{MaTa} (see also \cite[Eq. (18)]{Fino}) we have
$$
\frac{1}{32}\Lambda^2(dd^c \omega) = 2 d^*(Jd^*\omega) + 2
|d^*\omega|^2 - |d^c\omega|^2.
$$
Using now that $dd^c\omega = dH = 0$ and that $\nabla^-$ has holonomy
contained in $\SU(3)$---hence $\rho^B = 0$---we obtain
(cf. \cite{Fino})
\begin{equation}\label{eq:uom}
-2\Lambda_\omega dd^c \phi = 2|d^c\omega|^2 + 4|d^*\omega|^2 \geqslant 0.
\end{equation}
Integrating \eqref{eq:uom} over $M$, we conclude that $\omega$ is
K\"ahler, $H= 0$ and $\phi$ is constant. Therefore, $g$ is a metric
with holonomy contained in $\SU(3)$.

For the converse, we define $\eta = 1 \in \Lambda^{0,even}$ in the
model for $S_+(V_-)$ given by the $\SU(3)$-structure.  Then, since $g$
has holonomy $\SU(3)$ we have that $\nabla^- = \nabla^g$ preserves the
previous isomorphism and hence we obtain a solution of
\eqref{eq:Killingclassical} with $H= 0$ and $\phi$ constant.
\end{proof}

We go now for the transitive case, using the notation introduced in
Section \ref{sec:parameters}.  We make the assumption that the first
Pontryagin class of the principal bundle $P$ with respect to the
non-degenerate pairing on the Lie algebra of the structure group
$\mathfrak{g} = \mathfrak{k} \oplus \mathfrak{gl}(6,\RR)$
$$
c = 2\alpha'(- \tr_\mathfrak{k} - c_{\mathfrak{gl}})
$$
vanishes $p_1(P) = 0$ or, equivalently,
$$
p_1(P_K) = p_1(P_{\GL^+}).
$$
By \cite[Prop. 2.3]{GF}, this condition determines a canonical exact
Courant algebroid
$$
0 \to T^*P \to \hat E \to TP \to 0
$$
endowed with a (lifted) $G$-action and non-degenerate pairing $c$
(such that it admits an equivariant isotropic splitting). The
transitive Courant algebroid of our interest
$$
0 \to T^* \to E \to T \to 0,
$$
is then obtained from reduction \cite[Prop. 2.4]{GF} (alternatively,
one can apply \cite[Th. 1.7]{ChStXu} for a direct construction).

On $E$, we consider admissible metrics $V_+$ such that the metric $g$
on $M$ induced by $V_+$ is positive definite and the connection
$\theta$ on $P$ is a product of a connection $A$ on $P_K$ and a
$g$-compatible connection $\nabla$ on $P_{\GL^+}$.  With this ansatz,
the compatibility between $\theta$ and $H$ given in Proposition
\ref{prop:admissmetric2} reads
\begin{equation}\label{eq:bianchicharst}
  dH = 2\alpha'\(- c_{\mathfrak{gl}}(R\wedge R) -\tr_\mathfrak{k}(F_A\wedge F_A)\).
\end{equation}

We are ready to prove the main result of this section, which states
the equivalence of the Strominger system with the Killing spinor
equations.

\begin{proof}[Proof of Theorem \ref{th:charStrom}]
  Given a solution of \eqref{eq:Killing2}, from Theorem
  \ref{thm:Strom} we obtain a Calabi-Yau threefold structure $\Omega$
  on $M$ and a conformally balanced K\"ahler form $\omega$ with $H=
  d^c \omega$ .  Note that by \eqref{eq:bianchicharst} $H$ is not
  closed, and therefore the last part of the argument in the proof of
  Theorem \ref{th:SU3} does not apply.  Using Lemma
  \ref{lem:charStrom} and the first equation in
  \eqref{eq:Killingclassicaltr}, it follows from \cite{Wang} that
$$
F\wedge \omega^2 = 0, \qquad F \wedge \Omega = 0
$$
and hence both $A$ and $\nabla$ are hermitian-Yang-Mills connections.
Furthermore, 
since $\nabla g = 0$, the inclusion $\mathfrak{so}(6) \subset
\mathfrak{sl}(6,\RR)$ and \eqref{eq:bianchicharst} imply
$$
dd^c \omega = 2\alpha' (\tr R \wedge R - \tr_\mathfrak{k} F_A \wedge
F_A),
$$
by definition of $c_{\mathfrak{gl}}$. For the converse, given a
solution of the Strominger system we define $\theta = A \times
\nabla$, $H= d^c\omega$ and $\phi$ by \eqref{eq:phiSU3}. Then, the
spinor $\eta$ determined by the given $\SU(3)$-structure (see proof of
Theorem \ref{th:SU3}) on $M$ satisfies \eqref{eq:Killingclassicaltr}
and therefore is Killing.

The last part of the statement follows from Proposition
\ref{prop:naturality}.
\end{proof}

Theorem \ref{th:charStrom} shows that the Strominger system provides
natural equations in generalized geometry. As a direct consequence, we
obtain a precise geometric interpretation of the vector spaces
$H^1(\mathring{S}^*)$ and $H^1(\widehat{S}^*)$, as spaces of
infinitesimal deformations of solutions of the Killing spinor
equations \eqref{eq:Killing} modulo infinitesimal symmetries of a
Courant algebroid.


\begin{thebibliography}\frenchspacing\smallbreak

\bibitem{ACMM} M. C. Abbati, R. Cirelli, A. Mani\`a and P. W. Michor,
  \emph{The Lie group of automorphisms of a principal bundle},
  J. Geom. Phys. \textbf{6} (1989) 215--235.

\bibitem{AlexIvanov} B. Alexandrov and S. Ivanov, \emph{Vanishing
    theorems on hermitian manifolds}, Differential Geom. Appl.  {\bf
    14} (2001) 251--265.


\bibitem{AGS} L. B. Anderson, J. Gray and E. Sharpe, \emph{Algebroids,
    heterotic moduli spaces and the Strominger system}, JHEP {\bf
    07} (2014) 37.

\bibitem{AGF1} B. Andreas and M. Garcia-Fernandez, \emph{Solutions of
    the Strominger system via stable bundles on Calabi-Yau
    threefolds}, Commun. Math. Phys. \textbf{315} (2012), no. 1,
  153--168.

         





\bibitem{Bar} D.~Baraglia, \emph{Leibniz algebroids, twistings and
    exceptional generalized geometry}, J. Geom. Phys. {\bf 62} (2012)
  903--934.

\bibitem{BarHek} D.~Baraglia and P. Hekmati, \emph{Transitive Courant
    Algebroids, String Structures and T-duality}, Advances in Theoretical and Mathematical Physics {\bf 19} (2015), no. 3, 613--672.


\bibitem{BeckerTseng} K.~Becker and L. Tseng, \emph{Heterotic Flux
    Compactifications and their moduli}, Nucl. Phys. {\bf B 741}
  (2006) 162--179.
 		
\bibitem{BeckerTsengYau} K.~Becker, L. Tseng and S.-T- Yau,
  \emph{Moduli space of torsional manifolds}, Nucl. Phys. {\bf B 786}
  (2007) 119--134.

%


		
\bibitem{BerlineGetzler} N. Berline, E. Getzler and M. Vergne,
  \emph{Heat Kernels and Dirac Operators}, Grundlehren Text Editions,
  Springer (1992).






\bibitem{BuCaGu} H. Bursztyn, G. Cavalcanti and M. Gualtieri,
  \emph{Reduction of Courant algebroids and generalized complex
    structures}, Adv. Math. {\bf 211} (2) (2007) 726--765.
 		
\bibitem{Calabi} E. Calabi, \emph{The space of K\"ahler metrics},
  Proc. Int. Congr. Math. Amsterdam {\bf 2} (1954) 206--207.



\bibitem{CHSW} P. Candelas, G. Horowitz, A. Strominger, E. Witten,
  \emph{Vacuum Configurations for Superstrings}, Nucl. Phys. B {\bf
    258} (1985) 46--74.



\bibitem{ChStXu} Z. Chen, M. Stienon and P. Xu, \emph{On regular
    Courant algebroids}, J. Symp. Geom. {\bf 11} (2013) 1--24.

\bibitem{CSW} A. Coimbra, C. Strickland-Constable and D. Waldram,
  \emph{Supergravity as Generalised Geometry I: Type II Theories},
  JHEP {\bf 11} (2011) 91.
	
\bibitem{CoMiWa} A. Coimbra, R. Minasian, H. Triedl and D. Waldram,
  \emph{Generalized geometry for string corrections}, JHEP {\bf
    11} (2014) 160.
			
\bibitem{CyLa} M. Cyrier and J. M. Lapan, \emph{Towards the massless
    spectrum of non-K\"ahler heterotic compactifications},
  Adv. Theor. Math. Phys. {\bf 10} (2007) 853--877.

\bibitem{OssaSvanes} X. De la Ossa and E. Svanes, \emph{Holomorphic
    bundles and the moduli space of $N=1$ supersymmetric heterotic
    compactifications}, JHEP {\bf 10} (2014) 123.






\bibitem{Don} S.K.~Donaldson, \emph{Anti-self-dual Yang--Mills
    connections on a complex algebraic surface and stable vector
    bundles}, Proc. London Math. Soc. {\bf 50} (1985) 1--26.

\bibitem{DN} A.~Douglis and L. Niremberg, \emph{Interior estimates for
    elliptic systems of partial differential equations}, Comm. Pure
  App. Math. (4) {\bf 8} (1955) 503--538.
		
\bibitem{FeiYau} T. Fei and S.-T.~Yau, Invariant Solutions to the
  Strominger System on Complex Lie Groups and Their Quotients, Commun. Math. Phys. \textbf{338} (2015), no. 3, 1183--1195.

\bibitem{FIUVa} M. Fern\'andez, S. Ivanov, L. Ugarte, D. Vassilev,
  \emph{Non-K\"ahler heterotic string solutions with non-zero fluxes
    and non-constant dilaton}, JHEP {\bf 6} (2014) 73.

\bibitem{FIVU} M. Fern\'andez, S. Ivanov, L. Ugarte, R. Villacampa,
  \emph{Non-K\"ahler heterotic-string compactifications with non-zero
    fluxes and constant dilaton}, Commun. Math. Phys. {\bf 288} (2009)
  677--697. 

\bibitem{Fino} A. Fino and A. Tomassini, \emph{On astheno-K\"ahler
    metrics}, J. London Math. Soc. {\bf 83} (2011) 290--308.



\bibitem{Freed} D. Freed, \emph{Determinants, torsion and strings},
  Commun. Math. Phys. {\bf 107} (1986) 483--513.

\bibitem{Fu} J.-X.~Fu, \emph{On non-K\"ahler Calabi-Yau Threefolds
    with Balanced Metrics}, Proc. Int. Congress of Mathematicians,
  Hyderabad, India, Volume II, 705--716, Hindustan Book Agency, New Delhi (2010).

\bibitem{FuLiYau} J.-X.~Fu, J.~Li and S.-T.~Yau, \emph{Balanced metrics on non-K\"ahler Calabi-Yau threefolds}, J. Diff. Geom.  {\bf 90} (2012) 81--129.

\bibitem{FuYau} J.-X. Fu and S.-T.~Yau, \emph{The theory of
    superstring with flux on non-K\"ahler manifolds and the complex
    Monge-Amp\`ere equation}, J. Diff. Geom. {\bf 78} (2008)
  369--428. 

\bibitem{FTY} J.-X.~Fu, L.-S.~Tseng and S.-T.~Yau, \emph{Local
    heterotic torsional models}, Commun. Math. Phys. {\bf 289} (2009)
  1151--1169. 

\bibitem{FWW} J.-X.~Fu, Z.~Wang and D.~Wu, \emph{Form-Type Calabi-Yau Equations}, Math. Res. Lett. {\bf 17} (2010), 887--903.






\bibitem{GF} M. Garcia-Fernandez, \emph{Generalized connections and
    heterotic supergravity}, Commun. Math. Phys. {\bf 332} (2014)
  89--115.

\bibitem{GFT} M. Garcia-Fernandez, C. Tipler, {\it Deformations of
    complex structures and the coupled K\"ahler-Yang-Mills equations},
J. Lond. Math. Soc. {\bf 89} (2014), no. 3, 779--796.

\bibitem{Gau} P. Gauduchon, \emph{Hermitian connections and Dirac
    operators}, Bollettino U.M.I. (7) {\bf 11-B} (1997) 257--288.


\bibitem{Got} R. Goto, \emph{Moduli spaces of topological
    calibrations, Calabi-Yau, hyperK\"ahler, G2 and Spin(7)
    structures}, Int. J. Math. {\bf 15} (2004), no. 3, 211--257.

\bibitem{GreenSchwarz} M. B. Green and J. H. Schwarz, \emph{Anomaly
    cancellations in supersymmetric $D = 10$ gauge theory and
    superstring theory}, Phys. Lett. B {\bf 149} (1984) 117--122.


\bibitem{G1} M. Gualtieri, \emph{Generalized Complex Geometry}, DPhil thesis, University of Oxford (2004), arXiv:0401221.




\bibitem{G3} \bysame, \emph{Branes on Poisson varieties}, \emph{The
    many facets of geometry}, 368--394, Oxford Univ. Press, Oxford (2010). 


\bibitem{Hit1} \bysame, \emph{Generalized Calabi-Yau manifolds},
  Q. J. Math {\bf 54} (2003) 281--308.



\bibitem{hu} L. Huang, \emph{On joint moduli
    spaces}. Math. Ann. \textbf{302} (1995) 61--79.
    
\bibitem{HullTurin} C. Hull, \emph{Superstring compactifications with torsion and space-time supersymmetry}, In Turin 1985 Proceedings ``Superunification and Extra Dimensions'' (1986) 347--375.






\bibitem{Ivan09} S.~Ivanov, \emph{Heterotic supersymmetry, anomaly
    cancellation and equations of motion}, Phys. Lett. B {\bf 685}
  (2-3) (2010) 190--196.

\bibitem{IvanovPapadopoulos} S.~Ivanov and G.~Papadopoulos,
  \emph{Vanishing theorems and string backgrounds},
  Class. Quant. Grav. {\bf 18} (2001) 1089--1110.
		
\bibitem{IvanovPapadopoulosnogo} \bysame, \emph{A no-go theorem for
    string warped compactifications}, Phys.  Lett. B {\bf 497} (2001)
  309--316.






\bibitem{KaledinVerbitsky} D. Kaledin and M. Verbitsky,
  \emph{Non-hermitian Yang--Mills connections}, Selecta Mathematica
  {\bf 4} (1988) 279--320.


\bibitem{kim} H. J. Kim, {\it Curvatures and holomorphic vector
    bundles}, PhD thesis, University of California, Berkeley (1985).


\bibitem{Kob} S. Kobayashi, \emph{Differential Geometry of Complex
    Vector Bundles}, Princeton University Press (1987).

\bibitem{KNI} S. Kobayashi and K. Nomizu, \emph{Foundations of
    Differential Geometry}, Volume I, Interscience Publishers, New
  York (1963).


        
\bibitem{LaMi} H. Lawson and M. Michelsohn, \emph{Spin geometry}, {\bf
    38} of Princeton Mathematical Series, Princeton University Press,
  Princeton, N. J. (1989).
		
\bibitem{Hwasung} H. Lee, \emph{Strominger's System on non-K\"ahler
    hermitian Manifolds}, DPhil thesis, University of Oxford (2011).

\bibitem{LiYauHYM} J.~Li and S.-T.~Yau, \emph{hermitian-Yang-Mills
    connections on non-K\"ahler manifolds}, Mathematical aspects of
  string theory (San Diego, Calif., 1986) 560--573,
  Adv. Ser. Math. Phys., 1, World Sci. Publishing, Singapore (1987).

\bibitem{LiYau} J.~Li and S.-T.~Yau, \emph{The existence of
    supersymmetric string theory with torsion}, J. Diff. Geom. {\bf
    70} (2005) 143--181.


\bibitem{LM1} R. B. Lockhart and R. C. Mc Owen, \emph{On elliptic systems
  in $\RR^n$}, Acta Mathematica (1) {\bf 150} (1983), 125-135.

\bibitem{LM2} \bysame, \emph{Elliptic differential operators on non-compact
  manifolds}, Annali della Scuola Normale Superiore di Pisa {\bf 12}
  (3) (1985) 409--447.
        
\bibitem{lt} M. L\"ubcke and A. Teleman, {\it The Kobayashi-Hitchin
    correspondence}, World Scientific Publishing Co. Inc. (1995).

\bibitem{mac} K. Mackenzie, {\it General theory of Lie groupoids and
    Lie algebroids }, London Mathematical Society Lecture Note Series,
  no. 213, Cambridge University Press (2005). 



\bibitem{MaSp} D.~Martelli and J.~Sparks, \emph{Non-K\"ahler heterotic
    rotations}, Adv. Theor. Math. Phys. {\bf 15} (1) (2011), 131--174. 
		
\bibitem{MaTa} K. Matsuo and T. Takahashi, \emph{On compact
    astheno-K\"ahler manifolds}, Colloq. Math. {\bf 89} (1) (2001),
  213--221.
  
\bibitem{MelSha} I. Melnikov and E. Sharpe, \emph{On marginal
    deformations of $(0,2)$ non-linear sigma models}, Phys.
  Lett. {\bf B705} (2011) 529--534.

		
\bibitem{Michel} M. L. Michelsohn, \emph{On the existence of special
    metrics in complex geometry}, Acta Math. {\bf 149} (1) (1982),
  261--295.




\bibitem{popo} D. Popovici, {\it Holomorphic Deformations of Balanced
    Calabi-Yau $\del\delb$-Manifolds}, ArXiv preprint, arxiv:1304.0331.



\bibitem{Rubio} R.~Rubio, \emph{$B_n$-generalized geometry and
    $G_2^2$-structures}, J. Geom. Phys. {\bf 73} (2013) 150--156.
		
\bibitem{Rubioth} \bysame, \emph{Generalized geometry of type $B_n$},
 DPhil thesis, University of Oxford (2014).

\bibitem{SSS} H. Sati, U. Schreiber and J. Stasheff, \emph{Twisted differential string and fivebrane structures}, Commun. Math. Phys. {\bf 315} (2012) 169--213.








\bibitem{Strom} A.~Strominger, \emph{Superstrings with torsion},
  Nucl. Phys. B {\bf 274} (2) (1986) 253--284.


\bibitem{TsengYau} L.-S.~Tseng and S.-T.~Yau, \emph{Non-K\"ahler
    Calabi-Yau Manifolds}, String-Math 2011, 241--254, Proc. Symposia in Pure Mathematics {\bf 85} (2012).

\bibitem{UY} K.~Uhlenbeck and S.-T. Yau, \emph{On the existence of
    hermitian-Yang-Mills connections on stable bundles over compact
    K\"{a}hler manifolds}, Comm. Pure and Appl. Math. {\bf 39-S}
  (1986) 257--293.


\bibitem{Wang} M. Wang, \emph{Parallel spinors and parallel forms},
  Ann. Global Anal. Geom. (1) {\bf 7} (1989) 59--68.

\bibitem{Witt1} F. Witt, \emph{Special metric structures and closed
    forms}, DPhil thesis, University of Oxford (2005).

\bibitem{Witt2} \bysame, \emph{Calabi-Yau manifolds with B-fields},
  Rend. Sem. Mat. Univ. Politec. Torino {\bf 66} (2008) 1--21.






\bibitem{Yau0} S.-T. Yau, \emph{Calabi's conjecture and some new
    results in algebraic geometry}, Proc. Natl. Acad. Sci. USA {\bf
    74} (1977) 1798--1799.

\bibitem{Yau1} \bysame, \emph{Complex geometry: Its brief history and
    its future}, Science in China Series A Mathematics {\bf 48} (2005)
  47--60.

\bibitem{Yau2} \bysame, \emph{Metrics on complex manifolds},
  Sci. China Math. {\bf 53} (2010), no. 3, 565--572.

\end{thebibliography}
\end{document}